\documentclass[10pt,reqno]{amsart}
\usepackage{amssymb,mathrsfs,graphicx,amsmath,mathtools}
\usepackage{ifthen}
\usepackage{hyperref}
\usepackage[margin=1.2in]{geometry}
\usepackage{caption}
\usepackage{sidecap}
\usepackage{rotating}
\usepackage{color}
\DeclareMathOperator*{\esssup}{ess\,sup}

\provideboolean{shownotes} 
\setboolean{shownotes}{true} 
\newcommand{\margnote}[1]{
\ifthenelse{\boolean{shownotes}}%
{\marginpar{\raggedright\tiny\texttt{#1}}}%
{}%
}
\newcommand{\hole}[1]{
\ifthenelse{\boolean{shownotes}}%
{\begin{center} \fbox{ \rule {.25cm}{0cm} \rule[-.1cm]{0cm}{.4cm}
\parbox{.85\textwidth}{\begin{center} \texttt{#1}\end{center}} \rule
{.25cm}{0cm}}\end{center}} {} }

\title[Large friction-High force fields limit for the nonlinear VPFP system]{Large friction-High force fields limit for the nonlinear Vlasov--Poisson--Fokker--Planck system}

\author[Carrillo]{Jos\'{e} A. Carrillo}
\address[Jos\'{e} A. Carrillo]{\newline Mathematical Institute \newline 
University of Oxford, Oxford OX2 6GG, UK.}
\email{carrillo@maths.ox.ac.uk}

\author[Choi]{Young-Pil Choi}
\address[Young-Pil Choi]{\newline Department of Mathematics\newline
Yonsei University, 50 Yonsei-Ro, Seodaemun-Gu, Seoul 03722, Republic of Korea.}
\email{ypchoi@yonsei.ac.kr}

\author[Peng]{Yingping Peng}
\address[Yingping Peng]{\newline School of Mathematical Sciences \newline
University of Electronic Science and Technology of China, Chengdu, 611731, China. 
}
\email{yingping$\_$peng163.com}

\numberwithin{equation}{section}

\newtheorem{theorem}{Theorem}[section]
\newtheorem{lemma}{Lemma}[section]
\newtheorem{corollary}{Corollary}[section]
\newtheorem{proposition}{Proposition}[section]
\newtheorem{remark}{Remark}[section]
\newtheorem{definition}{Definition}[section]

\newcommand{\R}{\mathbb R}

\newcommand{\ls}{\lesssim}

\newcommand{\T}{\mathbb T}

\newcommand{\mc}{\mathcal C}

\newcommand{\bq}{\begin{equation}}
\newcommand{\eq}{\end{equation}}
\newcommand{\e}{\varepsilon}
\newcommand{\lt}{\left}
\newcommand{\rt}{\right}

\newcommand{\pa}{\partial}
\newcommand{\na}{\nabla}
\newcommand{\cd}{\cdot}
\newcommand{\ep}{\varepsilon}
\newcommand{\f}{\frac}

\newcommand{\mh}{\mathcal{H}}

\newcommand{\me}{\mathcal{E}}
\newcommand{\mk}{\mathcal{K}}
\newcommand{\mf}{\mathcal{F}}
\newcommand{\md}{\mathcal{D}}

\newcommand{\intr}{\int_{\R^d}}
\newcommand{\intrr}{\iint_{\R^d \times \R^d}}
\newcommand{\sfH}{\mathsf{H}}

\begin{document}
\allowdisplaybreaks

\date{\today}

\subjclass[]{}

\begin{abstract} We provide a quantitative asymptotic analysis for the nonlinear Vlasov--Poisson--Fokker--Planck system with a large linear friction force and high force-fields. The limiting system is a diffusive model with nonlocal velocity fields often referred to as aggregation-diffusion equations. We show that a weak solution to the Vlasov--Poisson--Fokker--Planck system strongly converges to a strong solution to the diffusive model. Our proof relies on the modulated macroscopic kinetic energy estimate based on the weak-strong uniqueness principle together with a careful analysis of the Poisson equation.

\end{abstract}

\maketitle \centerline{\date}


%
%
%
%
\section{Introduction}
We are concerned in this work with the asymptotic analysis of the nonlinear Vlasov--Poisson--Fokker--Planck (in short, VPFP) system with a linear damping in a high force-field regime. More precisely, let $f = f(x,v,t)$ be the particle, electron for instance, distribution function in the phase space $(x,v)$ at time $t > 0$. Our main system reads as
\begin{align}\label{main_kin0}
\begin{aligned}
&\pa_t f + v \cdot \nabla_x f - \frac{1}{m_e}\nabla_v \cdot (( v + \nabla_x V + \nabla_x \Phi)f) = \frac{1}{\tau_e}\mathcal{N}_{FP}(f), \quad (x,v,t) \in \R^d \times \R^d \times \R_+,\cr
&-\Delta_x \Phi = \rho, \quad \rho(x,t) = \intr f(x,v,t)\,dv,
\end{aligned}
\end{align}
where $V: \R^d \to \R_+$ represents the confinement potential, and $\mathcal{N}_{FP}$ denotes the nonlinear Fokker--Planck operator \cite{Vill02} given by
\[
\mathcal{N}_{FP}(f) := \nabla_v \cdot \lt((v - u)f + \sigma_e\nabla_v f\rt), \quad \mbox{where }   u := \frac{\intr vf\,dv}{\intr f\,dv}.
\]
Here the positive constants $m_e$ and $\tau_e$ represent the mass of the electrons and the relaxation time, respectively, and $\sqrt{\sigma_e} = \sqrt{(k_B T_{th})/m_e}$ denotes the thermal velocity,  with the Planck constant $k_B$ and the temperature of the thermal bath $T_{th}$. Other physical constants are taken to be unity for simplicity. The first term in $\mathcal{N}_{FP}(f)$ describes the nonlinear relaxation towards the local velocity $u$, and it can be rigorously derived from velocity alignment forces \cite{KMT14}. 

For the system \eqref{main_kin0}, we are interested in the common large friction and high force fields regime. To be more specific, we consider a small mass of electrons $m_e$ and a fast relaxation time $\tau_e$. In particular, for sufficiently small $\e \in (0,1)$, we choose $m_e = \e$, $T_{th} \simeq 1$, and $\tau_e = \e^{2+\delta}$ for some $\delta > 0$. Due to the relaxation between $m_e$ and $T_{th}$, we obtain $\sigma_e = k_B \e^{-1}$, and by setting $k_B = 1$, we deduce from \eqref{main_kin0} that
\begin{align}\label{main_kin}
\begin{aligned}
&\pa_t f + v \cdot \nabla_x f - \frac1\e\nabla_v \cdot (( v + \nabla_x V + \nabla_x \Phi)f) = \frac{1}{\e^{2+\delta}} \nabla_v \cdot \lt((v - u)f + \frac1\e\nabla_v f\rt),\cr
&-\Delta_x \Phi = \rho.
\end{aligned}
\end{align}
Here we omitted the $\e$-dependence of solutions $f$ for the sake of notational simplicity. For the rest of this paper, without loss of generality, we assume that the particle distribution function $f$ has a unit mass, i.e. $\iint_{\R^d \times \R^d} f\,dxdv = 1$ for all times due to the conservation of mass.

We study the asymptotic behavior of the nonlinear VPFP system \eqref{main_kin} as $\e \to 0$. To be more specific, we will show that the VPFP system \eqref{main_kin} converges towards the following drift-diffusion model with nonlocal velocity fields often referred to as {\it aggregation-diffusion equations}:
\bq\label{main_conti}
\pa_t \bar\rho + \nabla_x \cdot (\bar\rho \bar u) = 0, \quad  \bar\rho \bar u = -\bar \rho(\nabla_x V + \nabla_x \bar \Phi + \nabla_x \log \bar\rho), \quad -\Delta_x \bar\Phi = \bar\rho.
\eq
Note that if we define the free energy $E[\bar\rho]$ associated to the above by
\[
E[\bar\rho] = \intr \bar\rho \log \bar\rho\,dx + \intr V \bar\rho\,dx + \frac12 \intr \bar\Phi \bar\rho\,dx,
\]
the velocity field $ \bar u$ of equation \eqref{main_conti} can be rewritten as
\[
 \bar u = -\nabla_x \frac{\delta E}{\delta \bar\rho} [\bar\rho],
\]
where $ \frac{\delta E}{\delta \bar\rho}$ denotes the variation of the free energy $E$ with respect to the mass density $\bar\rho$. This shows that the equation \eqref{main_conti} has a gradient flow structure \cite{JKO98,CMV03,AGS05,CCY19}. 

%
%
%
%
\subsection{Formal derivation}
Let us explain how we can derive the equations \eqref{main_conti} as the limiting equation of \eqref{main_kin} as $\e \to 0$ at the formal level. If we define local particle density and moment by
\bq\label{def_ru}
\rho := \intr f\,dv \quad \mbox{and} \quad \rho u := \intr vf\,dv,
\eq
then by taking into account zero and first moments on $f$ of \eqref{main_kin}, we find that $\rho$ and $u$ satisfy
\begin{align}\label{local_cons}
\begin{aligned}
&\pa_t\rho + \nabla_x \cdot (\rho u) = 0,\cr
&\pa_t ( \rho u) + \nabla_x \cdot ( \rho  u\otimes u) + \frac1\e \nabla_x \rho = -\frac1\e \rho  u - \frac1\e\rho(\nabla_x V + \nabla_x \Phi) + e,\cr
&-\Delta_x \Phi = \rho,
\end{aligned}
\end{align}
where $e$ is given by
\bq\label{eq_e}
e := \nabla_x \cdot \lt( \intr \lt(u\otimes u - v \otimes v + \frac1\e \mathbb{I}\rt)f\,dv \rt).
\eq 
If we set a local Maxwellian
\[
M^\e_{u(x)}(v) := \frac{\e^{d/2}}{(2\pi)^{d/2}} \exp\lt( -\frac{\e|u(x)-v|^2}{2}\rt),
\]
then
\[
\intr M^\e_{u(x)}(v)\,dv = 1 \quad \mbox{for all } x \in \R^d
\]
and
\[
\mathcal{N}_{FP}(f) = \frac1\e \nabla_v \cdot \lt(f \nabla_v \log \frac{f}{M^\e_u} \rt).
\]
On the other hand, for $0<\e \ll 1$, we expect from \eqref{main_kin} that $\mathcal{N}_{FP}(f) \simeq 0$. This infers
\[
f(x,v) \simeq  \rho(x) M^\e_{u(x)}(v) = \frac{\e^{d/2}}{(2\pi)^{d/2}} \rho(x) \exp\lt( -\frac{\e|u(x)-v|^2}{2}\rt) \quad \mbox{for} \quad 0<\e \ll 1.
\]
Moreover we notice that
\begin{align*}
\intr (v \otimes v) \exp\lt( -\frac{\e|u(x)-v|^2}{2}\rt)\,dv&= \intr (u \otimes u) \exp\lt( -\frac{\e|u(x)-v|^2}{2}\rt)\,dv\cr
&\quad  + \intr (v - u)\otimes (v-u) \exp\lt( -\frac{\e|u(x)-v|^2}{2}\rt)\,dv\cr
&=  \lt(\frac{\e^{d/2}}{(2\pi)^{d/2}}\rt)^{-1} \lt((u \otimes u) + \frac1\e  \mathbb{I} \rt).
\end{align*}
Thus we obtain
\[
e \simeq \nabla_x \cdot \lt( \rho \intr \lt(u\otimes u - v \otimes v + \frac1\e \mathbb{I}\rt)M^\e_u\,dv \rt) =0
\]
for $0<\e \ll 1$. Putting this into \eqref{local_cons} yields
\begin{align*}
&\pa_t\rho + \nabla_x \cdot (\rho u) = 0,\cr
&\pa_t ( \rho u) + \nabla_x \cdot ( \rho  u\otimes u) + \frac1\e \nabla_x \rho = -\frac1\e \rho  u - \frac1\e\rho(\nabla_x V + \nabla_x \Phi),\cr
&-\Delta_x \Phi = \rho
\end{align*}
for $0 < \e \ll 1$, and again we deduce from the momentum equation in the above system that
\[
\rho u \simeq  - \rho(\nabla_x V + \nabla_x \Phi) - \nabla_x \rho = -\rho (\nabla_x V + \nabla_x \Phi + \nabla_x \log \rho)
\]
for $0 < \e \ll 1$. This reduces to our main limiting equations \eqref{main_conti}.

\begin{remark}\label{rem-1}The solution of the Poisson equation $\Phi$ can be uniquely expressed as a convolution $K \star \rho$, i.e. $\Phi = K \star \rho$, where the Coulomb interaction potential $K$ is explicitly given by
\[
K(x) = \left\{ \begin{array}{ll}
 -\frac{|x|}{2} & \textrm{for $d=1$,}\\[2mm]
 -\frac{1}{2\pi} \log |x| & \textrm{for $d=2$,}\\[2mm]
\frac{1}{(d-2)|B(0,1)|}\frac{1}{|x|^{d-2}} & \textrm{for $d \geq 3$},
  \end{array} \right.
\]
where $|B(0,1)|$ denotes the volume of unit ball $B(0,1)$ in $\R^d$, i.e. $|B(0,1)| = \pi^{d/2}/\Gamma(d/2+1)$. Here $\Gamma$ is the Gamma function.
\end{remark}
%
%
%
%
\subsection{Literature review}
In the absence of the nonlinear Fokker--Planck operator, the asymptotic analysis for the kinetic equation  \eqref{main_kin0} with regular interaction forces in the high force-field regime, i.e. \eqref{main_kin0} with $m_e = \e$ and $\mathcal{N}_{FP} \equiv 0$, is first studied in \cite{Jabin00} and refined in \cite{FS15}. To be more precise, the following aggregation equation can be rigorously derived from \eqref{main_kin0} with replacing $\nabla_x \Phi$ by $\nabla_x \tilde K \star \rho$ and taking $m_e = \e$ and $\mathcal{N}_{FP} \equiv 0$ as $\e \to 0$:
\bq\label{agg_eq}
\pa_t \bar\rho + \nabla_x \cdot (\bar\rho \bar u) = 0, \quad   \bar\rho \bar u = -\bar \rho(\nabla_x V + \nabla_x \tilde K  \star \bar\rho)
\eq
under suitable regularity assumptions on $\tilde K$ and $V$. In \cite{FS15,Jabin00}, types of compactness arguments are employed, and thus the convergences of solutions are only qualitatively investigated. On the other hand, in a recent paper \cite{CC20} a new idea of establishing the large friction limit of kinetic equation to \eqref{agg_eq} is proposed. Introducing an intermediate system, which is pressureless Euler-type system, and employing the second order Wasserstein distance between $\rho$ and $\bar\rho$, a quantified high force-field limit for the equation \eqref{main_kin0} with $m_e = \e$, $\tau_e  = \e$, and $\sigma_e=0$ is provided. However, a rather strong regularity for the interaction potential $W$ is still imposed, and thus singular interactions cannot be covered. 

In case $\sigma_e > 0$, i.e. with diffusive force, there are several works on the overdamped limit for the linear Vlasov--Fokker--Planck-type equation, i.e. it is a form of \eqref{main_kin0} with the linear diffusion $\Delta_v f$ instead of $\mathcal{N}_{FP}(f)$ on the right hand side. In \cite{DLPSS18, DLPS17}, the rigorous passage from the linear Vlasov--Fokker--Planck-type equation to the diffusive model is established by using a variational technique. In particular, in \cite{DLPSS18}, a quantitative estimate is obtained when there is no nonlocal interaction forces. On the other hand, nonlocal interactions are considered in \cite{DLPS17}, and a qualitative error estimate is investigated based on compactness arguments combined with duality methods.  Very recently, in \cite{CTpre} the quantified overdamped limit for the Vlasov--Poisson--Fokker--Planck equation with linear diffusion and singular interaction forces is analysed. In particular, either attractive or repulsive Coulomb potential can be taken into account, and the equation \eqref{main_conti} with the associated velocity fields is rigorously derived. The argument used in \cite{CTpre} relies on the evolution-variational-like inequality for Wasserstein gradient flows, and the convergence of $\rho$ towards $\bar\rho$ is obtained in the Wasserstein distance of order 2. We also refer to \cite{CCT19,CPW20,CG07,LT13,LT17} for the rigorous derivation of aggregation-diffusion-type equations from Euler-type system through a singular limit and \cite{GNPS05,NPS01,PS00} for the other singular limits of the VPFP system with linear diffusion, namely the parabolic and hyperbolic scalings.
%
%
%
%
\subsection{Contribution}
In the present work, we establish quantitative strong convergences of the kinetic equation \eqref{main_kin} towards the aggregation-diffusion equation \eqref{main_conti} when $\e \to 0$. We identify the asymptotic regime where the VPFP system \eqref{main_kin0} well approximates the aggregation-diffusion equation \eqref{main_conti}; weak solutions to the system \eqref{main_kin} strongly converge towards the unique strong solution to the equation \eqref{main_conti}. We would like to emphasize that these quantitative strong convergences have not been investigated so far, up to our best knowledge. We show the $L^1$-convergence of $(\rho, \rho u)$ towards $(\bar\rho, \bar\rho \bar u)$, in particular this implies the almost everywhere convergence up to a subsequence.  We also would like to emphasize that the asymptotic regime treated here differs substantially from the high field or hyperbolic scaling regime typically considered to study the asymptotic analysis of Vlasov--Fokker--Planck-type equations with linear relaxation operator \cite{GNPS05,NPS01}. To the best of our knowledge, up to now it has not been observed that the derivation of aggregation-diffusion equations can be established via the combined large friction-high force-field limit. 

In \cite{CTpre,DLPSS18, DLPS17}, an intermediate equation via a coarse-graining map is introduced for the overdamped limit for the Vlasov--Fokker--Planck-type equation. On the other hand,  in case $\sigma_e = 0$ as mentioned above the pressureless Euler-type system is considered as the associated intermediate system for \eqref{main_kin0} in \cite{CC20}. However, these strategies are not available for our problem due to the nonlocal alignment force term in the operator $\mathcal{N}_{FP}$. One can think of the isothermal Euler-type system as the intermediate system and follow the methodology proposed in \cite{CC20}. In this case, it seems hard to have the uniform-in-$\e$ Lipschitz estimate on the velocity field due to the presence of pressure, see \cite[Section 3.1]{CC20}. For that reason, we directly estimate the error between two equations \eqref{main_kin} and \eqref{main_conti} without introducing the intermediate system. We first reformulate the limiting equation \eqref{main_conti} as the conservative form, isothermal Euler-type system, with an error term. We then employ the relative entropy method. This resembles the strategy used in \cite{CPW20,LT13,LT17}, however careful analysis on the entropy and the error term is needed. In particular, in \cite{CPW20,LT17}, $L^\infty$-bound assumption on $\bar u$ is used, and this makes some estimates regarding the error term comfortable. However, due to the presence of the confinement potential $V$, it seems impossible to impose that assumption in our case when we consider the quadratic confinement potential $V=|x|^2/2$. In order to handle this issue, we observe a cancellation structure and use the entropy inequality, which requires higher-order regularity of solutions. See Lemma \ref{lem_ee} and Remark \ref{rmk_gd} below for more detailed discussion. Moreover, we provide the required existence theories and needed estimates for the equations \eqref{main_kin} and \eqref{main_conti} to make all of our results self-contained and fully rigorous. 
%
%
%
%
\subsection{Notation}
Let us introduce several notations used throughout the present work. For functions, $f(x,v)$ and $g(x)$, $\|f\|_{L^p}$ and $\|g\|_{L^p}$ represent the usual $L^p(\R^d \times \R^d)$-and $L^p(\R^d)$-norms, respectively. We denote by $C$ a generic positive constant which is independent of $\e$. $C = C(\alpha,\beta, \cdots)$ stands for a positive constant depending on $\alpha,\beta, \cdots$. $f \ls g$ represents that there exists a positive constant $C > 0$ such that $f \leq Cg$. For simplicity of notation, we often drop $x$-dependence of differential operators, i.e. $\nabla f:= \nabla_x f$ and $\Delta f := \Delta_x f$ whenever there is no any confusion. For any nonnegative integer $k$ and $p \in [1,\infty]$, $W^{k,p} = W^{k,p}(\R^d)$ stands for the $k$-th order $L^p$ Sobolev space. In particular, if $p=2$, we denote by $H^k=H^k(\R^d) = W^{k,2}(\R^d)$. $\mc^k([0,T];E)$ is the set of $k$-times continuously differentiable functions from an interval $[0,T] \subset \R$ into a Banach space $E$, and $L^p(0,T;E)$ is the set of measurable functions from an interval $(0,T)$ to a Banach space $E$, whose $p$-th power of the $E$-norm is Lebesgue measurable.  $\nabla^k$ stands for any partial derivative $\pa^\alpha$ with multi-index $\alpha$, $|\alpha| = k$. 
 
%
%
%
%
\subsection{Main result}

We first introduce notions of solutions to the VPFP equations \eqref{main_kin} and the aggregation-diffusion equations \eqref{main_conti}  below.

\begin{definition}\label{def_weak} For given $T \in (0,\infty)$, $f$ is a weak solution of \eqref{main_kin} on the time interval $[0,T]$ if and only if the following conditions are satisfied:
\begin{itemize}
\item[(i)] $f \in L^\infty(0,T; (L^1_+ \cap L^\infty)(\R^d \times \R^d))$,
\item[(ii)] for any $\varphi \in \mc^\infty_c(\R^d \times \R^d \times [0,T])$,
\begin{align*}
&\int_0^t  \intrr f\lt(\pa_s \varphi + v \cdot \nabla \varphi - \frac1\e(v + \nabla V + \nabla \Phi) \cdot \nabla_v \varphi \rt) dxdvds\cr
&\quad +  \frac{1}{\e^{2+\delta}}\int_0^t  \intrr f\lt( (u-v) \cdot \nabla_v \varphi + \frac1\e \Delta_v \varphi \rt)\,dxdvds = \intrr f_0 \varphi(x,v,0)\,dxdv.
\end{align*}
\end{itemize}
Here $L^1_+(\R^d \times \R^d)$ denotes a set of nonnegative $L^1(\R^d \times \R^d)$-functions.
\end{definition}

\begin{definition}\label{def_strong} For given $T \in (0,\infty)$, $\bar\rho$ is a strong solution of \eqref{main_conti} on the time interval $[0,T]$ if and only if the following conditions are satisfied:
\begin{itemize}
\item[(i)] $\bar\rho  \in \mc([0,T]; L^1_V(\R^d)) \cap L^\infty(0,T;W^{1,1} \cap W^{1,\infty}(\R^d))$, $\nabla \log \bar\rho \in L^\infty(0,T;W^{2,\infty}(\R^d))$,
\item[(ii)] $(\bar\rho , \bar u)$ satisfies \eqref{main_conti} in the sense of distributions.
\end{itemize}
Here $L^1_V(\R^d)$ denotes the space of weighted measurable functions by $1+V$ with the norm
\[
\|\bar\rho\|_{L^1_V}:= \intr (1 + V) \bar\rho\,dx.
\]
\end{definition}

\begin{remark} The regularity assumption on $\bar\rho$ in Definition \ref{def_strong} (i) can be replaced by
\[
\bar\rho  \in \mc([0,T]; L^1_V(\R^d)) \cap L^\infty(0,T;W^{3,1} \cap W^{3,\infty}(\R^d)),
\]
see Section \ref{sec_rem} for details. 
\end{remark}

Since the global-in-time existence of weak solution for the nonlinear VPFP equations \eqref{main_kin0} in the sense of Definition \ref{def_weak} is obtained in \cite{CCJpre}, we do not give any details on that here. We refer to \cite{YP16,KMT13} for the global-in-time existence of weak and strong solutions for \eqref{main_kin0} with $V \equiv 0$ and $\Phi \equiv 0$. For the existence theory for the Vlasov--Poisson equation with the linear diffusion, i.e. \eqref{main_kin0} with $\Delta_v f$ instead of $\mathcal{N}_{FP}(f)$, we refer to \cite{B93,CS95,V91,VO90}. Qualitative properties and existence of solutions to aggregation-diffusion equations have been studied under different assumptions \cite{CMV03, AGS05,CCY19} and the references therein.

We next define a modulated energy $\sfH_\e = \sfH_\e((\rho,u)|(\bar\rho, \bar u))$ by
\[
\sfH_\e((\rho,u)|(\bar\rho, \bar u)) := \f12\intr\rho|u-\bar{u}|^2\,dx  +   \f{1}{\e}\intr \int_{\bar \rho}^{\rho} \frac{\rho - z}{z}\,dzdx + \f{1}{2\e}\intr|\na (\Phi - \bar\Phi)|^2 \,dx.
\]

We now state our main theorem.

\begin{theorem}\label{thm_main}
Let $T>0$. Let $f$ be the solution to the equation \eqref{main_kin} in the sense of Definition \ref{def_weak} and $\bar\rho$ be the strong solution to the equation \eqref{main_conti} in the sense of Definition \ref{def_strong} on the time interval $[0,T]$. Denote by 
$$
M_0=\intr \lt(\intr f_0 \frac{|v|^2}{2}\,dv - \rho_0|u_0|^2\rt)dx
$$
and
$$
\bar M_0=  \intr  \lt(\intr f_0\log f_0\,dv - \rho_0 \log\rho_0\rt)dx\,,
$$
the nonnegative difference between the mesoscopic and the macroscopic initial kinetic energy and entropy for \eqref{main_kin}. Then for $\e \in (0,1)$ small enough we have the following estimates.
\begin{itemize}
\item[(i)] Case with the confinement ($V(x) = |x|^2/2$): 
\begin{align*}
 \sfH_\e(t) +  \frac{1}{2\e} \int_0^t \intr \rho|u-\bar u|^2\,dxds \leq &\, C \sfH_\e(0)  + C M_0+C\e^{\delta}\mathcal{F}_{\e}(f_0) +   C\e^{\delta-d/2-1}    +    C\e  + \frac C\e \bar M_0.   
\end{align*}
In particular, this implies
\begin{align*}
\begin{aligned}
&\intr \int_{\bar \rho}^{\rho} \frac{\rho - z}{z}\,dz\,dx  +   \intr|\na (\Phi - \bar\Phi)|^2\, dx  +   \int_0^t \intr \rho|u-\bar u|^2\,dxds
 \cr
 &\qquad\qquad \leq C\e\sfH_\e(0)   +C\e M_0+C\e^{\delta+1}e^{2\e^{\delta}T}\mathcal{F}_{\e}(f_0)  +   C\e^{\delta-d/2}    +    C\e^2 +  C \bar M_0.   
\end{aligned}
\end{align*}
\item[(ii)] Case without the confinement ($V \equiv 0$): 
\begin{align*}
\sfH_\e(t) +  \frac{1}{4\e} \int_0^t \intr \rho|u-\bar u|^2\,dxds \leq C \sfH_\e(0)+ C M_0+C\ep^{2+2\delta}E_\e(f_0) +    C\e   +   C\e^{\delta-1} + \frac{C}{\ep} \bar M_0 .
\end{align*}
In particular, this implies
\begin{align*}
\begin{aligned}
&\intr \int_{\bar \rho}^{\rho} \frac{\rho - z}{z}\,dz\,dx  +   \intr|\na (\Phi - \bar\Phi)|^2\, dx  +   \int_0^t \intr \rho|u-\bar u|^2\,dxds\cr
&\qquad\qquad \leq  C\e\sfH_\e(0)   +C\e M_0      +C\ep^{3+2\delta}E_\e(f_0)    + C\e^2  +   C\e^{\delta} +  C \bar M_0.   
\end{aligned}
\end{align*}
\end{itemize}

Here $C>0$ is independent of $\e > 0$,
\[
\mathcal{F}_\e(f_0) := \frac1\e \intrr f_0\log f_0\,dxdv + \intrr f_0 \frac{|v|^2}{2}\,dxdv + \frac{1}{\e} \intr  \lt(V+ \frac12 \Phi_0\rt) \rho_0\,dx,
\]
and
\bq\label{def_e}
E_\e(f_0):=\intrr f_0\frac{|v|^2}{2}dxdv+\frac{1}{2\ep}\intr\rho_0\Phi_0\,dx.
\eq
\end{theorem}

\begin{corollary}\label{cor_main}

Suppose that all the assumptions in Theorem \ref{thm_main} hold. Furthermore we assume that 
\[
\intr \lt(\intr f_0 \frac{|v|^2}{2}\,dv - \rho_0|u_0|^2\rt)dx \leq C
\]
and
\[
\intr \int_{\bar \rho_0}^{\rho_0} \frac{\rho_0 - z}{z}\,dzdx +  \intr  \lt(\intr f_0\log f_0\,dv - \rho_0 \log\rho_0\rt)dx + \intr|\na (\Phi_0 - \bar\Phi_0)|^2\,dx \leq C\e^\zeta
\]
for some $C>0$ independent of $\e>0$, where $\zeta$ is given by
\[
\zeta := \left\{ \begin{array}{ll}
 \min\{1,\delta-d/2\} & \textrm{for $V(x) = |x|^2/2$}\\[2mm]
 \min\{1,\delta\}  & \textrm{for $V \equiv 0$}
  \end{array} \right..
\]
Then we have
\[
\intr \int_{\bar \rho}^{\rho} \frac{\rho - z}{z}\,dz\,dx  +   \intr|\na (\Phi - \bar\Phi)|^2\, dx  +   \int_0^t \intr \rho|u-\bar u|^2\,dxds \leq C\e^\zeta
\]
for $t < T$, where $C>0$ is independent of $\e$. In particular, if $\zeta >0$, then the following strong convergences hold:
\[
\rho \to \bar\rho \quad \mbox{in} \quad L^\infty(0,T;L^1 \cap H^{-1}(\R^d))
\]
and
\[
\rho u \to \bar \rho \bar u \quad \mbox{in} \quad L^2(0,T;L^1(\R^d))
\]
as $\e \to 0$.
\end{corollary}

\begin{remark} One may extend our result to the case with confinement potentials $V$ satisfying
\[
e^{-V} \in L^1(\R^d), \quad |\nabla^2 V(x)| \leq c_1, \quad \mbox{and} \quad |\nabla V(x)|^2 \leq c_2(1 + V(x)) \quad \mbox{for } x \in \R^d,
\] 
for some $c_i > 0$, $i=1,2$. It is clear that the quadratic confinement potential $V(x) = |x|^2/2$ satisfies the above inequality with $c_1 = 1$ and $c_2 = 2$. Our strategy can be easily extended to the case that $\Phi$ is given as $\nabla \tilde K \star \rho$, where the interaction potential $\tilde K$ satisfying $\nabla \tilde K \in L^\infty(\R^d)$. See Remark \ref{rmk_linf} for details.
\end{remark}

%
%
%
%
\subsection{Organization of paper}
The rest of this paper is organized as follows. In Section \ref{sec:rel}, we provide a free energy estimate and introduce our main functional, relative entropy functional. Section \ref{sec:main} is devoted to prove our main results; Theorem \ref{thm_main} and Corollary \ref{cor_main}. Finally, in Section \ref{sec:bdd}, we present some bound estimates for the solution $\bar\rho$ to \eqref{main_conti} which are required for our quantitative error estimates.

%
%
%
%
%
%
%
%
\section{Preliminaries}\label{sec:rel}

\subsection{Free energy estimate}
Let us first introduce free energy $\mf_\e$ for the system \eqref{main_kin} and its associated dissipation $\md_\e$:
\[
\mathcal{F}_\e(f) := \frac1\e \intrr f\log f\,dxdv + \intrr f \frac{|v|^2}{2}\,dxdv + \frac1\e \intr \rho V\,dx + \frac{1}{2\e} \intr \Phi \rho\,dx
\]
and
\[
\mathcal{D}_\e(f) := \intrr \frac{1}{f} \lt|\frac1\e \nabla_v f - (u-v)f\rt|^2\,dxdv.
\]

\begin{lemma}\label{lem_ent}Let $T>0$ and $f$ be a solution of \eqref{main_kin} with sufficient integrability on the time interval $[0,T]$. Then we have
\[
\mathcal{F}_\e(f)  +   \int_0^t\left(\f{1}{\ep^{2+\delta}}\mathcal{D}_\e(f)   +  \f{1}{\ep}\intrr f|v|^2dxdv\right)ds    \le  \mathcal{F}_\e(f_0)  + d \e^{-2}t.
\]
for $t \in [0,T]$. Furthermore, we obtain
\[
\mathcal{F}_\e(f)  +   \int_0^t\left(\f{1}{\ep^{2+\delta}}\mathcal{D}_\e(f)   +  \f{1}{\ep}\intr\rho|u|^2\,dx\right)ds    \le  \mathcal{F}_\e(f_0)  +  \f{\e^{\delta}}{2}\int_0^t\intrr f|v|^2\,dxdvds.
\]
Here $\rho$ and $u$ are defined as in \eqref{def_ru}.
\end{lemma}
\begin{proof} Using equation \eqref{main_kin}, we can directly compute 
\begin{align}\label{entropy-f-1}
\begin{aligned}
\f{d}{dt}\intrr f\log f\, dxdv
&= \intrr (1 + \log f)\, \pa_t f\, dxdv    \cr
&= -  \f{1}{\ep}\intrr \f{1}{f}\na_v f\cd\Big(v    +   (\na V  +  \na \Phi)\Big)\,f\, dxdv   \cr
&\quad \,  -  \f{1}{\ep^{2+\delta}}\intrr \f{1}{f}\na_v f\cd \Big((v-u)f   +  \f{1}{\ep}\na_v f\Big)\,dxdv     \cr
&= \f{d}{\ep}    -   \f{1}{\ep^{2+\delta}}\intrr\f{1}{f}\na_v f\cd \Big((v-u)f   +  \f{1}{\ep}\na_v f\Big)\,dxdv.
\end{aligned}
\end{align}
Similarly, we deduce 
\begin{align*}
\f{d}{dt}\intrr f\f{|v|^2}{2}dxdv &=  \intrr \pa_t f \f{|v|^2}{2} dxdv   \cr
&=  -\f{1}{\ep}\intrr f|v|^2dxdv    -   \f{d}{dt}\left(\f{1}{\ep}\intr \rho V dx
   +  \f{1}{2\ep}\intr\Phi\rho\,dx  \right)     \cr
&\quad -\f{1}{\ep^{2+\delta}}\intrr f|v-u|^2dxdv    -   \f{1}{\ep^{3+\delta}}\intrr(v-u)\cd\na_v f\, dxdv,
\end{align*}
which subsequently implies 
\begin{align}\label{entropy-f-2}
\begin{aligned}
&\f{d}{dt}\left(\intrr f\f{|v|^2}{2}\,dxdv   +   \f{1}{\ep}\intr\rho V\, dx
   +  \f{1}{2\ep}\intr\Phi\rho\,dx  \right)     \cr
&\quad  =  - \f{1}{\ep}\intrr f|v|^2\,dxdv -\f{1}{\ep^{2+\delta}}\intrr f|v-u|^2\,dxdv    \cr
&\qquad - \f{1}{\ep^{3+\delta}}\intrr (v-u)\cd\na_v f\, dxdv.
\end{aligned}
\end{align}
Dividing \eqref{entropy-f-1} by $\ep$ and adding the resulting equation to \eqref{entropy-f-2}, we get
\bq\label{entropy-f-3}
\f{d}{dt}\mathcal{F}_\e(f) +   \f{1}{\ep}\intrr f|v|^2dxdv
 +  \f{1}{\ep^{2+\delta}}\intrr \f{1}{f}\left|\f{1}{\ep}\na_v f - (u-v)f \right|^2 dxdv    =   \f{d}{\ep^2}.
\eq
Integrating the above with respect to time gives the first assertion. 

One can further estimate
\begin{align}\label{entropy-f-4}
\begin{aligned}
\f{d}{\ep^2} & =  \f{1}{\ep}\intrr v\cd \left((u-v)f  -  \f{1}{\ep}\na_v f\right)dxdv
    -\f{1}{\ep}\intrr v\cd(u-v)f\,dxdv   \cr
&\le \f{1}{\ep}\left(\intrr f|v|^2\,dxdv\right)^{1/2}
   \left(\intrr \f{1}{f}\left|\f{1}{\ep}\na_v f-(u-v)f\right|^2dxdv\right)^{1/2}\cr
   &\quad
   -\f{1}{\ep}\intr \rho|u|^2\,dx   +   \f{1}{\ep}\intrr f|v|^2 \,dxdv     \cr
&\le  \f{1}{2\ep^{2+\delta}}\mathcal{D}_\e(f) 
    +    \f{\e^{\delta}}{2}\intrr f|v|^2\,dxdv  -\f{1}{\ep}\intr \rho|u|^2dx   +   \f{1}{\ep}\intrr f|v|^2 \,dxdv.
\end{aligned}
\end{align}
Substituting \eqref{entropy-f-4} into \eqref{entropy-f-3}, we have
\begin{align*}
\f{d}{dt}\mathcal{F}_\e(f)   +   \f{1}{2\ep^{2+\delta}}\mathcal{D}_\e(f)   +  \f{1}{\ep}\intr\rho|u|^2\,dx
\le      \f{\e^{\delta}}{2}\intrr f|v|^2\,dxdv
\end{align*}
from which, we obtain the second assertion.
\end{proof}

\begin{remark} In \cite{CCJpre}, it is showed that the weak solutions $f$ of \eqref{main_kin} in the sense of Definition \ref{def_weak} satisfies the entropy inequality appeared in Lemma \ref{lem_ent}.
\end{remark}

\subsection{Relative entropy}

Note that the equation \eqref{main_conti} can be also rewritten as
\begin{align}\label{euler}
\begin{aligned}
&\pa_t\bar\rho + \nabla \cdot (\bar\rho \bar u) = 0,\cr
&\pa_t (\bar \rho \bar u) + \nabla \cdot (\bar \rho \bar u\otimes \bar u) + \frac1\e \nabla \bar \rho = -\frac1\e \bar\rho \bar u - \frac1\e\bar \rho(\nabla V + \nabla \bar\Phi) + \bar e,\cr
&-\Delta \bar\Phi = \bar\rho,
\end{aligned}
\end{align}
where $\bar e = \bar\rho (\pa_t \bar u + \bar u \cdot \nabla \bar u)$.
Let us rewrite the system \eqref{euler} as a conservative form:
\[
\pa_t \bar U + \nabla \cdot A_{\e}(\bar U) = F_{\e}(\bar U),
\]
where
\[
\bar m = \bar \rho \bar u, \quad \bar U := \begin{pmatrix}
\bar \rho \\
\bar m
\end{pmatrix},
\quad
A_\e(\bar U) := \begin{pmatrix}
\bar m  & 0 \\
(\bar m \otimes \bar m)/\bar \rho & \bar \rho /\e
\end{pmatrix},
\]
and
\[
F_\e(\bar U) := \begin{pmatrix}
0 \\
\displaystyle -\frac1\e \bar\rho \bar u - \frac1\e\bar \rho(\nabla V + \nabla \bar\Phi) + \bar e
\end{pmatrix}.
\]
Then the above system has the following macroscopic entropy form:
\[
\me_\e(\bar U) := \frac{|\bar m|^2}{2\bar\rho} + \frac1\e\bar\rho \log \bar\rho.
\]
We now define the relative entropy functional $\mh_\e$ as follows.
\[
\mh_\e(U|\bar U) := \me_\e( U) - \me_\e(\bar U) - D\me_\e(\bar U)( U-\bar U) \quad \mbox{with} \quad U := \begin{pmatrix}
        \rho \\
         m \\
    \end{pmatrix}, \quad  m = \rho  u,
\]
where $D \me_\e(\bar U)$ denotes the derivation of $\me_\e$ with respect to $\bar\rho, \bar m$, and we find
\begin{align*}
-D\me_\e(\bar U)( U - \bar U) &= -\begin{pmatrix}
\displaystyle        -\frac{|\bar m|^2}{2\bar\rho^2} + \frac1\e(\log \bar \rho + 1)\\[3mm]
\displaystyle        \frac{\bar m}{\bar\rho}
    \end{pmatrix}
    \begin{pmatrix}
    \rho -\bar \rho \\
    m - \bar m
    \end{pmatrix}\\
    &= \frac{\rho |\bar u|^2}{2} - \frac{\bar \rho|\bar u|^2}{2} + \frac1\e(\bar\rho - \rho)(\log \bar\rho + 1) + \bar\rho |\bar u|^2 - \rho \bar u \cdot u.
\end{align*}
This yields
\[
\mh_\e(U| \bar U)  = \frac{\rho}{2}|u - \bar u|^2 + \frac1\e p(\rho| \bar \rho),
\]
where $p$ represents the relative entropy which is defined by
\[
p(\rho| \bar  \rho) := \int_{\bar \rho}^{\rho} \frac{\rho - z}{z}\,dz. 
\]

%
%
%
%
%
\section{Proofs of Theorem \ref{thm_main} \& Corollary \ref{cor_main}}\label{sec:main}

\subsection{Relative entropy estimate}
\begin{proposition}\label{rel_prop}Let $T>0$. Let $f$ be the solution to the equation \eqref{main_kin} in the sense of Definition \ref{def_weak} and $\bar\rho$ be the strong solution to the equation \eqref{main_conti} in the sense of Definition \ref{def_strong} on the time interval $[0,T]$. Then for $\e \in (0,1)$ small enough we have
\begin{align}\label{rel}
\begin{aligned}
&\f12\intr\rho|u-\bar{u}|^2\,dx  +   \f{1}{\e}\intr p(\rho|\bar{\rho})\,dx  +  \frac{1}{2\e} \int_0^t \intr \rho|u-\bar u|^2\,dxds
  +   \f{1}{2\e}\intr|\nabla (\Phi - \bar\Phi)|^2 \,dx \cr
 &\quad \leq C\intr\rho_0|u_0-\bar{u}_0|^2\,dx  +   \f{C}{\e}\intr p(\rho_0|\bar{\rho}_0)\,dx +C\intr \lt(\mk_\e(f_0) - \me_\e(U_0)\rt)dx \cr
 &\qquad  \,\,+ C\e^{\delta}\int_0^t\intrr f|v|^2\,dxdv ds +    C\e +  \f{C}{\e}\intr|\nabla (\Phi_0 - \bar\Phi_0)|^2\,dx,
\end{aligned}
\end{align}
where $C>0$ is independent of $\e$.
\end{proposition}
\begin{proof}
Straightforward computations yield
\begin{align*}
\begin{aligned}
\intr \mh_\e(U|\bar U)\,dx &= \intr \mh_\e(U_0|\bar U_0)\,dx + \intr  \me_\e(U)\,dx - \intr \me_\e(U_0)\,dx \cr
&\quad -  \int_0^t\intr \nabla (D\me_\e(\bar U)):A_\e(U|\bar U)\,dxds \cr
&\quad -  \int_0^t \intr D^2\me_\e(\bar U) F_\e(\bar U) (U - \bar U) + D\me_\e(\bar U) F_\e(U)\,dxds\cr
&=: \sum_{i=1}^5 I_i,
\end{aligned}
\end{align*}
where we easily estimate $I_4$ as
\[
I_4 \leq \|\nabla \bar u\|_{L^\infty}\int_0^t \intr \rho|u - \bar u|^2\,dxds.
\]
Set $\mk_\e$ the mesoscopic entropy:
\[
\mk_\e(f) := \intr f \frac{|v|^2}{2}\,dv + \frac1\e \intr f\log f\,dv.
\]
By the classical minimization principle, see \cite{KMT15}, we find
\[
\intr \me_\e(U)\,dx \leq \intr \mk_\e(f)\,dx,
\]
and this gives
\begin{align*}
I_2+I_3 &= \intr \me_\e(U)\,dx - \intr \mk_\e(f)\,dx + \intr \mk_\e(f)\,dx - \intr \mk_\e(f_0)\,dx\cr
&\quad  + \intr \mk_\e(f_0)\,dx - \intr \me_\e(U_0)\,dx\cr
&\leq \intr \mk_\e(f)\,dx - \intr \mk_\e(f_0)\,dx + \intr \mk_\e(f_0)\,dx - \intr \me_\e(U_0)\,dx.
\end{align*}
On the other hand, since
\[
F_\e(U) = \begin{pmatrix}
0 \\
\displaystyle -\frac1\e \rho  u - \frac1\e \rho(\nabla V + \nabla \Phi) +  e
\end{pmatrix},
\]
 where $e$ is appeared in \eqref{eq_e}, we get
\begin{align}\label{est_f}
\begin{aligned}
&-\intr D^2\me_\e(\bar U) F_\e(\bar U) (U - \bar U) + D\me_\e(\bar U) F_\e(U)\,dx\cr
&\quad = -\frac1\e \intr \rho|u-\bar u|^2\,dx + \frac1\e \intr \rho|u|^2\,dx + \frac1\e \intr \rho u \cdot \nabla V\,dx + \frac1\e \intr \rho u \cdot \nabla \Phi\,dx\cr
&\qquad + \frac1\e \intr \rho (u - \bar u) \cdot \nabla (\bar\Phi - \Phi)\,dx - \intr \rho( u - \bar u) \cdot \frac{\bar e}{\bar \rho}\,dx - \intr \bar u \cdot e\,dx\cr
&\quad = -\frac1\e \intr \rho|u-\bar u|^2\,dx + \frac1\e \intr \rho|u|^2\,dx + \frac1\e \frac{d}{dt}\intr \rho  V\,dx + \frac1{2\e}\frac{d}{dt} \intr \rho \Phi\,dx\cr
&\qquad + \frac1\e \intr \rho (u - \bar u) \cdot \nabla (\bar\Phi - \Phi)\,dx - \intr \rho( u - \bar u) \cdot \frac{\bar e}{\bar \rho}\,dx - \intr \bar u \cdot e\,dx.
\end{aligned}
\end{align}
We next recall from \cite[Lemma 5.1]{YP20}, see also \cite{CFGS17, LT13, LT17} that
\[
\f12\f{d}{dt}\intr\lt|\na (\Phi - \bar\Phi)\rt|^2dx
=\intr\nabla (\Phi - \bar\Phi)\cd(\rho u - \bar{\rho}\bar{u})\,dx.
\]
Thus, we obtain
\begin{align}\label{Cou-2}
\begin{aligned}
&\frac1\e \intr \rho (u - \bar u) \cdot \nabla (\bar\Phi - \Phi)\,dx   \cr
&\quad =-\f{1}{\ep}\intr\nabla (\Phi - \bar\Phi)\cd(\rho u - \bar{\rho}\bar{u})\,dx
  + \f{1}{\ep}\intr\nabla (\Phi - \bar\Phi)\cd\bar{u}(\rho-\bar{\rho})\,dx  \cr
&\quad =  -\f1{2\ep}\f{d}{dt}\intr\lt|\nabla (\Phi - \bar\Phi)\rt|^2dx
 + \f{1}{\ep}\intr\nabla (\Phi - \bar\Phi)\cd\bar{u}(\rho-\bar{\rho})\,dx.
\end{aligned}
\end{align}
For the second term on the right-hand-side of \eqref{Cou-2}, we compute that
\begin{align}\label{Cou-3}
\begin{aligned}
&\f{1}{\ep}\intr\nabla (\Phi - \bar\Phi)\cd\bar{u}(\rho-\bar{\rho})\,dx \cr
&\quad = - \f{1}{\ep}\intr\nabla (\Phi - \bar\Phi)\cd\bar{u}\Delta (\Phi - \bar\Phi)\,dx   \cr
&\quad = -\f{1}{2\ep}\intr\lt|\nabla (\Phi - \bar\Phi)\rt|^2\nabla\cd\bar{u}\,dx  +   \f{1}{\ep}\intr\nabla (\Phi - \bar\Phi)\otimes\nabla (\Phi - \bar\Phi):\nabla\bar{u}\,dx    \cr
&\quad\le \f{3}{2\ep}\|\nabla\bar{u}\|_{L^{\infty}}\intr\lt|\nabla (\Phi - \bar\Phi)\rt|^2\,dx.
\end{aligned}
\end{align}
Substituting \eqref{Cou-3} into \eqref{Cou-2}, we see that
\begin{align*}
\frac1\e \intr \rho (u - \bar u) \cdot \nabla (\bar\Phi - \Phi)\,dx  \le   -\f1{2\ep}\f{d}{dt}\intr\lt|\nabla (\Phi - \bar\Phi)\rt|^2 dx
  +   \f{3}{2\ep}\|\nabla\bar{u}\|_{L^{\infty}}\intr\lt|\nabla (\Phi - \bar\Phi)\rt|^2 dx.
\end{align*}

For the estimate of the sixth term  on the right hand side of \eqref{est_f}, we give the following lemma. As mentioned in Introduction, we cannot use the assumption $\bar u \in L^\infty(\R^d \times (0,T))$ when we consider $V=|x|^2/2$, and thus we estimate it differently from \cite{CPW20,LT17}. For the sake of the reader, we provide the details of proof at the end of this subsection.

\begin{lemma}\label{lem_ee} There exists a constant $C>0$ depending only on the regularity estimates $\|\bar\rho\|_{L^\infty(0,T;W^{1,1} \cap W^{1,\infty})}$ and $\|\nabla \log \bar\rho\|_{L^\infty(0,T;W^{2,\infty})}$ such that 
\[
\intr \rho( u - \bar u) \cdot \frac{\bar e}{\bar \rho}\,dx \leq \intr \rho |u - \bar u|^2\,dx + C\lt(\intr \rho |u - \bar u|^2\,dx \rt)^{1/2}.
\]
In particular, for $\e>0$ small enough we have
\[
\intr \rho( u - \bar u) \cdot \frac{\bar e}{\bar \rho}\,dx \leq \frac{1}{2\e}\intr \rho |u - \bar u|^2\,dx + C\e
\]
for some $C>0$ independent of $\e>0$.
\end{lemma}
We next estimate
\begin{align*}
\lt|\intr \bar u \cdot e\,dx\rt|&= \lt|\intr \bar u \cdot \lt(\nabla \cdot \lt( \intr \lt(u\otimes u - v \otimes v + \frac1\e \mathbb{I}\rt)f\,dv \rt)\rt) dx\rt|\cr
&\leq \|\nabla \bar u\|_{L^\infty} \intr \lt| \intr \lt(u\otimes u - v \otimes v + \frac1\e \mathbb{I}\rt)f\,dv\rt|dx\cr
&\leq C\|\nabla \bar u\|_{L^\infty}\lt( \intrr f|v|^2\,dxdv \rt)^{1/2}  \lt( \intrr \frac{1}{f}\lt| \frac1\e \nabla_v f - (u-v)f\rt|^2 dxdv\rt)^{1/2}\cr
&\leq C\|\nabla \bar u\|_{L^\infty}\e^{2+ \delta}
\intrr f|v|^2\,dxdv + 
\frac1{2\e^{2+\delta}}\mathcal{D}_\e(f) 
\end{align*}
due to \cite[Proof of Lemma 4.4]{KMT15}.
This yields
\begin{align*}
	\begin{aligned}
I_5 &\leq -\frac1{2\e} \int_0^t \intr \rho|u-\bar u|^2\,dxds -\f1{2\ep}\intr\lt|\nabla (\Phi - \bar\Phi)\rt|^2 dx+ \f1{2\ep}\intr\lt|\nabla (\Phi_0 - \bar\Phi_0)\rt|^2 dx \cr
&\quad   + C T\e   + \frac{C}{\e} \int_0^t \intr\lt|\nabla (\Phi - \bar\Phi)\rt|^2 dxds     + \frac1\e \int_0^t \intr \rho|u|^2\,dxds  \\ 
&\quad + C\e^{2+ \delta} \int_0^t\intrr f|v|^2\,dxdv ds +  \frac1{2\e^{2+\delta}}\int_0^t \mathcal{D}_\e(f)\,ds \cr
&\quad+ \frac1\e \lt(\intr \rho  V\,dx - \intr \rho_0  V\,dx\rt)  + \frac1{2\e}\lt( \intr \rho \Phi\,dx - \intr \rho_0  \Phi_0\,dx\rt).
\end{aligned}
\end{align*}
Thus we obtain
\begin{align*}
I_2 +I_3+ I_5 &\leq -\frac1{2\e} \int_0^t \intr \rho|u-\bar u|^2\,dxds -\f1{2\ep}\intr\lt|\nabla (\Phi - \bar\Phi)\rt|^2 dx + \f1{2\ep}\intr\lt|\nabla (\Phi_0 - \bar\Phi_0)\rt|^2 dx\cr
&\quad + \frac{C}{\e} \int_0^t \intr\lt|\nabla (\Phi - \bar\Phi)\rt|^2 dxds + \intr \mk_\e(f_0)\,dx - \intr \me_\e(U_0)\,dx   +    C T\e   \cr
&\quad + C\e^{\delta}\int_0^t\intrr f|v|^2\,dxdv ds.  
\end{align*}
Combining all of the above estimates provides
\begin{align}\label{Cou-6}
\begin{aligned}
&\f12\intr\rho|u-\bar{u}|^2\,dx  +   \f{1}{\e}\intr p(\rho|\bar{\rho})\,dx +  \frac{1}{2\e} \int_0^t \intr \rho|u-\bar u|^2\,dxds
  +   \f{1}{2\e}\intr|\nabla (\Phi - \bar\Phi)|^2 \,dx \cr
 &\qquad \leq \f12\intr\rho_0|u_0-\bar{u}_0|^2\,dx  +   \f{1}{\e}\intr p(\rho_0|\bar{\rho}_0)\,dx + \intr \lt(\mk_\e(f_0) - \me_\e(U_0)\rt)dx\cr
&\quad \qquad     +    C\e +  \f{C}{\e}\intr |\nabla (\Phi_0 - \bar\Phi_0)|^2\,dx
    +  \f{C}{\e}\int_0^t\intr |\nabla (\Phi - \bar\Phi)|^2\,dxds    \\
&\quad\qquad  \,\,+ C\e^{\delta}\int_0^t\intrr f|v|^2\,dxdv ds
\end{aligned}
\end{align}
for $\e \in (0,1)$, which implies that
\begin{align*}
\begin{aligned}
&\intr |\nabla (\Phi - \bar\Phi)|^2 \,dx \cr
&\quad  \leq \e\intr \rho_0|u_0-\bar{u}_0|^2\,dx   + 2\e\intr \lt(\mk_\e(f_0) - \me_\e(U_0)\rt)dx  + C\e^2   + C\e^{\delta+1}\int_0^t\intrr f|v|^2\,dxdv ds      \\
&\qquad    +   2\intr p(\rho_0|\bar{\rho}_0)\,dx +  C\intr |\nabla (\Phi_0 - \bar\Phi_0)|^2\,dx
+  C\int_0^t\intr |\nabla (\Phi - \bar\Phi)|^2\,dxds.
\end{aligned}
\end{align*}
Applying Gr\"onwall's lemma entails that
\begin{align*}
&\intr |\nabla (\Phi - \bar\Phi)|^2 \,dx   \\
&\quad  \leq C\e\intr\rho_0|u_0-\bar{u}_0|^2\,dx  +   C\intr p(\rho_0|\bar{\rho}_0)\,dx + C\e\intr \lt(\mk_\e(f_0) - \me_\e(U_0)\rt)dx  \cr
&\qquad   \,\,+ C\e^{\delta+1}\int_0^t\intrr f|v|^2\,dxdv ds   + C\e^2   +  C\intr|\nabla (\Phi_0 - \bar\Phi_0)|^2\,dx.
\end{align*}
Putting this into \eqref{Cou-6}, we have
$$
\begin{aligned}
&\f12\intr\rho|u-\bar{u}|^2\,dx  +   \f{1}{\e}\intr p(\rho|\bar{\rho})\,dx  +  \frac{1}{2\e} \int_0^t \intr \rho|u-\bar u|^2\,dxds
  +   \f{1}{2\e}\intr|\nabla (\Phi - \bar\Phi)|^2 \,dx \cr
 &\quad \leq C\intr\rho_0|u_0-\bar{u}_0|^2\,dx  +   \f{C}{\e}\intr p(\rho_0|\bar{\rho}_0)\,dx +C\intr \lt(\mk_\e(f_0) - \me_\e(U_0)\rt)dx \cr
 &\qquad  \,\,+ C\e^{\delta}\int_0^t\intrr f|v|^2\,dxdv ds +    C\e +  \f{C}{\e}\intr|\nabla (\Phi_0 - \bar\Phi_0)|^2\,dx
\end{aligned}
$$
for $\e \in (0,1)$, where $C>0$ is independent of $\e$. This completes the proof.
\end{proof}
\begin{proof}[Proof of Lemma \ref{lem_ee}]  We first notice that
\[
\bar u_j = -x _j - \pa_j \bar\Phi - \pa_j \log \bar\rho \quad \mbox{for} \quad j=1,\dots,d,
\]
and thus
\[
\bar u_i \pa_i \bar u_j = -\bar u_i \lt( \delta_{ij} + \pa_{ij} \bar\Phi  + \pa_{ij} \log \bar\rho\rt) \quad \mbox{for} \quad i,j=1,\dots,d.
\]
For notational simplicity, for the rest of this proof, we omit the summation, i.e., $u_i v_i = \sum_{i=1}^d u_i v_i$ and we denote by $\pa_i = \pa_{x_i}$ for $i=1,\dots,d$. Then we obtain
\begin{align}\label{rest_1}
\begin{aligned}
 \intr \rho (u - \bar u) \cdot (\bar u \cdot \nabla \bar u)\,dx
&= - \intr \rho (u_j - \bar u_j) \bar u_j \,dx -  \intr \rho (u_j - \bar u_j) \bar u_i (\pa_{ij} \bar\Phi ) \,dx\cr
&\quad -  \intr \rho (u_j - \bar u_j) \bar u_i \pa_{ij}\log \bar\rho \,dx.
\end{aligned}
\end{align}
On the other hand, we also find (see Remark \ref{rem-1})
\[
-\pa_j \pa_t \bar\Phi =  -\pa_j K \star \pa_t \bar\rho = \pa_{ij} K \star (\bar \rho \bar u_i) = \intr \pa_{ij} K(x-y) (\bar\rho \bar u_i)(y)\,dy
\]
and
\[
-\pa_j \pa_t \log \bar\rho = \pa_j (\bar u_i \pa_i \log \bar\rho + \pa_i \bar u_i) = (\pa_j \bar u_i) \pa_i \log \bar\rho + \bar u_i \pa_{ij} \log \bar\rho + \pa_{ij} \bar u_i
\]
for $j = 1,\dots,d$.
This implies
\begin{align*}
\intr \rho (u - \bar u) \cdot \pa_t \bar u\,dx&= \intrr \rho(x) (u_j - \bar u_j)(x) \pa_{ij}K(x-y)(\bar\rho \bar u_i)(y)\,dxdy \cr
&\quad + \intr \rho (u_j - \bar u_j)\lt( (\pa_j \bar u_i) \pa_i \log \bar\rho + \bar u_i \pa_{ij} \log \bar\rho + \pa_{ij} \bar u_i \rt) dx.
\end{align*}
We now combine this with \eqref{rest_1} to get
\begin{align*}
\intr \rho( u - \bar u) \cdot \frac{\bar e}{\bar \rho}\,dx
&= \intr \rho( u - \bar u) \cdot (\pa_t \bar u + \bar u \cdot \nabla \bar u)\,dx\cr
&= - \intr \rho (u_j - \bar u_j) \bar u_j \,dx -  \intr \rho (u_j - \bar u_j) \bar u_i (\pa_{ij} \bar\Phi ) \,dx\cr
&\quad + \intrr \rho(x) (u_j - \bar u_j)(x) \pa_{ij}K(x-y)(\bar\rho \bar u_i)(y)\,dxdy \cr
&\quad + \intr \rho (u_j - \bar u_j)\lt( (\pa_j \bar u_i) \pa_i \log \bar\rho  + \pa_{ij} \bar u_i \rt) dx\cr
&=: \sum_{i=1}^4 J_i,
\end{align*}
where $J_i, i=1,4$ can be easily estimated as
\begin{align*}
J_1 &= \intr \rho |u - \bar u|^2\,dx - \intr \rho ( u - \bar u) \cdot u\,dx\cr
& \leq \intr \rho |u - \bar u|^2\,dx + \lt(\intr \rho |u - \bar u|^2\,dx \rt)^{1/2} \lt(\intr \rho|u|^2\,dx \rt)^{1/2}\cr
&\leq \intr \rho |u - \bar u|^2\,dx + C\lt(\intr \rho |u - \bar u|^2\,dx \rt)^{1/2}
\end{align*}
and
\begin{align*}
J_4 &\leq \lt(\|\nabla (\nabla \cdot \bar u)\|_{L^\infty} + \|\nabla \bar u\|_{L^\infty}\|\nabla \log \bar\rho\|_{L^\infty} \rt) \intr \rho | u - \bar u|\,dx \cr
&\leq \lt(\|\nabla (\nabla \cdot \bar u)\|_{L^\infty} + \|\nabla \bar u\|_{L^\infty}\|\nabla \log \bar\rho\|_{L^\infty} \rt) \lt(\intr \rho | u - \bar u|^2\,dx\rt)^{1/2}.
\end{align*}
Here we used
\[
\intr \rho\,dx = 1 \quad \mbox{and} \quad \intr \rho|u|^2\,dx \leq \intrr f|v|^2\,dxdv \leq C
\]
for some $C>0$ independent of solutions $(\bar\rho, \bar u)$ and $\e$ due to Lemma \ref{lem_ent}. We next estimate
\begin{align*}
\lt|J_2 + J_3\rt| &= \lt|\intrr \rho(x) (u_j - \bar u_j)(x) \pa_{ij}K(x-y)\bar\rho(y) (\bar u_i(y) - \bar u_i(x))\,dxdy\rt|\cr
&\leq \|\nabla \bar u\|_{L^\infty} \intr \rho(x) |(u - \bar u)(x)| \lt(\intr \frac{1}{|x-y|^{d-1}} \bar\rho(y)\,dy \rt)\,dx\cr
&\leq C\|\nabla \bar u\|_{L^\infty}\|\bar\rho\|_{L^1 \cap L^\infty} \intr \rho |u - \bar u|\,dx\cr
&\leq C\|\nabla \bar u\|_{L^\infty}\|\bar\rho\|_{L^1 \cap L^\infty}\lt( \intr \rho |u - \bar u|^2\,dx\rt)^{1/2},
\end{align*}
where $C>0$ is independent of solutions $(\bar\rho, \bar u)$ and $\e$. 

We next rewrite the assumption on the regularity of solutions in terms of $\bar \rho$ only, not $\bar u$. In fact, this is not that hard, simply we find
\[
\nabla \bar u = -I_d - \nabla K \star \nabla \bar\rho - \nabla^2 \log \bar\rho
\]
and
\[
\nabla \cdot \bar u = -d + \bar \rho - \Delta \log \bar\rho.
\]
This yields
\[
\|\nabla \bar u\|_{L^\infty} \leq C\lt(1 + \|\nabla \bar\rho\|_{L^1 \cap L^\infty} + \|\nabla^2 \log \bar\rho\|_{L^\infty} \rt)
\]
and
\[
\|\nabla (\nabla \cdot \bar u)\|_{L^\infty} \leq \|\nabla \bar\rho\|_{L^\infty} + \|\nabla \Delta \log \bar\rho\|_{L^\infty}.
\]
This completes the proof.
\end{proof}

\subsection{Proof of Theorem \ref{thm_main}}
In order to close the relative entropy estimate in Proposition \ref{rel_prop}, we need to handle the kinetic energy term on the right hand side of \eqref{rel}. 

\subsubsection{Case with the confinement}

In this case, we show that the kinetic energy can be controlled by the free energy $\mf_\e(f)$. For this, we need to estimate the negative part of the entropy term. 

Note that there exists a positive constant $C$ such that the following estimate
\begin{align}\label{flogf-1}
g|\log g|  =  g\log g   -   2g\log g\,\chi_{0\leq g\leq1} \leq g\log g   +  2\left(\omega g+Ce^{-\omega/2}\right)
\end{align}
holds for $g,\omega\geq0$, where $\chi$ is a characteristic function.  We then take $g=f$ and $\omega=\frac{\varepsilon|v|^2+|x|^2}{8}$ in \eqref{flogf-1} to have
\begin{align*}
\intrr f|\log f|\,dxdv  \leq \intrr f\log f\,dxdv   +   \intrr \frac{\varepsilon|v|^2+|x|^2}{4}f\,dxdv    +   \frac{C}{\varepsilon^{d/2}}.
\end{align*}
This implies
\begin{align*}
\frac{1}{\varepsilon}\intrr f|\log f|\,dxdv  
&\leq \frac{1}{\varepsilon}\intrr f\log f\,dxdv   +   \intrr \frac{|v|^2}{4}f\,dxdv  \\
&\quad  + \frac{1}{\varepsilon}\intrr \frac{|x|^2}{4}f\,dxdv   +   \frac{C}{\varepsilon^{d/2+1}}.
\end{align*}
Thus we obtain
	\begin{align*}
		\begin{aligned}
			\mathcal{F}_{\varepsilon}(f)
			&= \frac1\e \intrr f\log f\,dxdv + \intrr f \frac{|v|^2}{2}\,dxdv + \frac1\e \intrr f\frac{|x|^2}{2}\,dx + \frac{1}{2\e} \intr \Phi \rho\,dx \\
			&\geq \frac1\e\intrr f|\log f|\,dxdv   +   \frac{1}{4}\intrr f|v|^2\,dxdv  \\
			&\quad  +   \frac{1}{4\e}\intrr f|x|^2\,dxdv   +   \frac{1}{2\e}\intr \Phi\rho \,dx   -    \frac{C}{\varepsilon^{d/2+1}},
		\end{aligned}
	\end{align*}
	and subsequently this together with Lemma \ref{lem_ent} yields
\[
\intrr f|v|^2\,dxdv  \leq 4\mathcal{F}_{\e}(f)   +   \frac{C}{\varepsilon^{d/2+1}} \leq \lt(4\mathcal{F}_{\e}(f_0)   +   \frac{C}{\varepsilon^{d/2+1}}\rt) + 2\e^\delta \int_0^t\intrr f|v|^2\,dxdvds.
\]
We now apply Gr\"onwall's lemma to the above to have
\[
 \int_0^t\intrr f|v|^2\,dxdvds \leq e^{-2\e^\delta} \lt(4\mathcal{F}_{\e}(f_0)   +   \frac{C}{\varepsilon^{d/2+1}}\rt)\lt( e^{2 \e^\delta t} - 1\rt) \leq C\mathcal{F}_{\e}(f_0) + C\e^{-d/2-1},
\]
where we used $4\mathcal{F}_{\e}(f_0) + C\e^{-d/2-1} \geq 0$. This combined with Proposition \ref{rel_prop} concludes the desired result.
\begin{remark}\label{rmk_gd} If we further assume $\bar u \in L^\infty(\R^d \times (0,T))$ i.e. $\bar e = \pa_t \bar{u} + \bar{u} \cdot\nabla \bar{u} \in L^\infty(\R^d \times (0,T))$, by the arguments used in \cite{CPW20, LT17}, then we can easily estimate
\begin{align*}
\intr \rho( u - \bar u) \cdot \frac{\bar e}{\bar \rho}\,dx &\leq \frac{1}{8\e}\intr \rho|u - \bar u|^2\,dx + C \e\intr \rho \lt|\pa_t \bar{u} + \bar{u} \cdot\nabla \bar{u} \rt|^2 dx\cr
&\leq \frac{1}{8\e}\intr \rho|u - \bar u|^2\,dx + C \e \|\bar e\|_{L^\infty}^2\|\rho\|_{L^1},
\end{align*}
where $C > 0$ is independent of $\e>0$. However, in the case with confinement, the velocity fields $\bar u$ is given by $\bar u = - x - \nabla \bar\Phi - \nabla \log \bar\rho$, and thus it seems impossible to assume the $L^\infty$-bound on $\bar u$.
\end{remark}

\subsubsection{Case without the confinement} 
Differently from the case with the confinement, in this case, we first control the (mesoscopic) kinetic energy by using the (macroscopic) kinetic and interaction energies. 

Following a similar way as in \eqref{entropy-f-2}, we deduce 
\begin{align}\label{entropy-ff}
\begin{aligned}
&\f{d}{dt}\left(\intrr f\f{|v|^2}{2}\,dxdv  +  \f{1}{2\ep}\intr\Phi\rho\,dx  \right)  \cr  
&\quad  + \f{1}{\ep}\intrr f|v|^2\,dxdv    +\f{1}{\ep^{2+\delta}}\intrr f|v-u|^2\,dxdv  
=\frac{d}{\ep^{3+\delta}}.
\end{aligned}
\end{align}
Integrating \eqref{entropy-ff} from $0$ to $t$, one has 
\begin{align*}
\begin{aligned}
\int_0^t\intrr f|u-v|^2\,dxdvds
\leq \ep^{2+\delta}E_\e(f_0) +   \frac{dT}{\ep},
\end{aligned}
\end{align*}
where $E_\e(f_0)$ is given as in \eqref{def_e}. Since 
\begin{align*}
\intrr f|u-v|^2\,dxdv   =  \intrr f|v|^2\,dxdv   -   \intr \rho|u|^2\,dx,
\end{align*}
we further have 
\[
\int_0^t\intrr f|v|^2\,dxdvds  
\leq  \ep^{2+\delta}E_\e(f_0)  +   \frac{C}{\ep}   +   \int_0^t\intr \rho|u|^2\,dxds.
\]
On the other hand, the (macroscopic) kinetic energy can be estimated as 
\begin{align*}
\intr\rho|u|^2\,dx  &\leq  2\intr\rho|u-\bar{u}|^2\,dx    +  2\intr\rho|\bar{u}|^2\,dx   \cr 
&\leq  \frac{1}{4\e^{\delta+1}}\intr\rho|u-\bar{u}|^2\,dx  +  2\|\nabla (\bar\Phi + \log \bar\rho)\|_{L^\infty}
\end{align*}
for $\e \in (0,1)$ small enough.  Combining this and Proposition \ref{rel_prop} completes the proof.

\subsection{Proof of Corollary \ref{cor_main}} By Taylor's theorem, we first easily find
\[
\int_{\bar \rho}^{\rho} \frac{\rho - z}{z}\,dz \geq \frac12 \min\lt\{\frac{1}{\bar\rho}, \frac{1}{\rho} \rt\}(\rho - \bar\rho)^2.
\]
We then estimate
\begin{align}\label{est_l1}
\begin{aligned}
\lt(\intr |\rho - \bar\rho|\,dx\rt)^2 &\leq \lt(\intr (\rho + \bar \rho) \,dx \rt)\lt(\intr \min\lt\{\frac{1}{\bar\rho}, \frac{1}{\rho} \rt\}(\rho - \bar\rho)^2\,dx \rt) \cr
&\leq C\intr \int_{\bar \rho}^{\rho} \frac{\rho - z}{z}\,dzdx,
\end{aligned}
\end{align}
where $C>0$ is independent of $\e>0$ and we used $1 \leq (x+y)\min\{x^{-1}, y^{-1}\}$ for $x,y > 0$. This asserts  the convergence of $\rho \to \bar\rho$ in $L^\infty(0,T;L^1(\R^d))$ as $\e \to 0$.

Moreover, for any $\psi \in H^1(\R^d)$ with $\|\psi\|_{H^1} \leq 1$, we obtain
\begin{align*}
\lt|\intr \psi(x)(\rho - \bar\rho)(x)\,dx \rt| &= \lt|\intr \psi(x)\Delta(\Phi - \bar\Phi)(x)\,dx \rt| \cr
&=\lt|\intr \nabla \psi(x)\cdot\nabla(\Phi - \bar\Phi)(x)\,dx \rt|\cr
&\leq \|\nabla (\Phi - \bar\Phi)\|_{L^2},
\end{align*}
and this yields
\[
\|\rho - \bar\rho\|_{H^{-1}} \leq \|\nabla (\Phi - \bar\Phi)\|_{L^2}.
\]
Thus we have the convergence of $\rho \to \bar\rho$ in $L^\infty(0,T;H^{-1}(\R^d))$ as $\e \to 0$.

For the convergence of $\rho u$ towards $\bar\rho \bar u$, we estimate
\begin{align*}
\intr |\rho - \bar \rho||\bar u|\,dx 
&\leq  \intr |\rho - \bar \rho| |\nabla V|\,dx +  \|\nabla (\bar\Phi + \log \bar\rho)\|_{L^\infty}  \intr |\rho - \bar \rho|\,dx\cr
&\leq C\lt(\intr |\rho - \bar \rho|\,dx\rt)^{1/2}\lt(\intr (\rho + \bar\rho)(1+V)\,dx\rt)^{1/2} + C\intr |\rho - \bar \rho|\,dx,
\end{align*}
where $C>0$ depends on $ \|\nabla (\bar\Phi + \log \bar\rho)\|_{L^\infty}$. This deduces
\begin{align*}
\intr |\rho u - \bar\rho \bar u|\,dx &\leq \intr \rho|u - \bar u|\,dx + \intr |\rho - \bar \rho||\bar u|\,dx\cr
&\leq \lt(\intr \rho\,dx \rt)^{1/2}\lt( \intr \rho|u - \bar u|^2\,dx\rt)^{1/2} + C\intr |\rho - \bar \rho|\,dx\cr
&\quad + C\lt(\intr |\rho - \bar \rho|\,dx\rt)^{1/2}\lt(\intr (\rho + \bar\rho)(1+V)\,dx\rt)^{1/2} .
\end{align*}
We finally combine this with \eqref{est_l1} to have
\[
\|\rho u - \bar\rho \bar u\|_{L^1}^2 \leq C\intr \rho|u - \bar u|^2\,dx + C\lt(\intr \int_{\bar \rho}^{\rho} \frac{\rho - z}{z}\,dzdx\rt)^{1/2} + C\intr \int_{\bar \rho}^{\rho} \frac{\rho - z}{z}\,dzdx,
\]
where $C>0$ depends on   $ \|\nabla (\bar\Phi + \log \bar\rho)\|_{L^\infty}$. Integrating the above inequality over the time interval $[0,T]$ concludes the desired convergence estimate.

\begin{remark}\label{rmk_linf} If $\Phi$ is given as $\Phi = \tilde K \star \rho$ with $\nabla \tilde K \in L^\infty(\R^d)$, then we can estimate the fifth term on the right hand side of \eqref{est_f} as
\begin{align*}
 \frac1\e \lt|\intr \rho (u - \bar u) \cdot \nabla \tilde K \star (\bar\rho - \rho)\,dx\rt|
 &\leq \frac1\e \|\nabla \tilde K\|_{L^\infty} \|\rho - \bar\rho\|_{L^1} \intr \rho |u - \bar u|\,dx\cr
 &\leq \frac C\e \lt(\intr \int_{\bar \rho}^{\rho} \frac{\rho - z}{z}\,dzdx\rt)^{1/2}\lt(\intr \rho |u - \bar u|^2\,dx \rt)^{1/2}\cr
 & \leq \frac C\e \intr \int_{\bar \rho}^{\rho} \frac{\rho - z}{z}\,dzdx + \frac1{4\e}\intr \rho |u - \bar u|^2\,dx .
\end{align*}
This combined with almost the same arguments as before concludes the same convergence estimates appeared in Corollary \ref{cor_main}.
\end{remark}


\subsection{Remarks on the regularity assumptions on the limiting system}\label{sec_rem}

In this part, we provide some estimates on $\|\nabla \log \bar\rho\|_{L^\infty(0,T;W^{2,\infty})}$ when $V = |x|^2/2$. The estimates directly cover the case $V \equiv 0$. Let us start with the estimate of $\|\nabla \log \bar\rho\|_{L^\infty}$. For notational simplicity, we set $\tilde u:= -\nabla V - \nabla \bar\Phi$, we again omit the summation, and we denote by $\pa_i = \pa_{x_i}$ for $i=1,\dots,d$.

We first show the $L^\infty$-bound on $\nabla \log \bar\rho$ in the lemma below.

\begin{lemma} There exists $C>0$ depends only on $\|\nabla \bar\rho\|_{L^\infty(0,T;L^1 \cap L^\infty)}$ and $T>0$ such that
\[
\sup_{0 \leq t \leq T} \| \nabla \log \bar\rho(\cdot,t)\|_{L^\infty} \leq C(1 +  \| \nabla \log \bar\rho_0\|_{L^\infty}).
\]
\end{lemma}
\begin{proof}
Note that $\pa_i \log \bar\rho$ satisfies
\[
\pa_t \pa_i \log \bar\rho + \tilde u \cdot \nabla \pa_i \log \bar\rho = - \nabla \cdot \pa_i \tilde u - \pa_i \tilde u \cdot \nabla \log \bar\rho + \Delta \pa_i \log \bar\rho + 2 \nabla \pa_i \log \bar\rho \cdot \nabla \log \bar\rho
\]
for $i=1,\dots,d$. For a given $t$, at any local Maximum point of $\pa_i \log \bar\rho$, we get
\[
\Delta \pa_i \log \bar\rho \leq 0 \quad \mbox{and} \quad \nabla \pa_i \log \bar\rho = 0.
\]
Using this observation together with an elementary estimate yields
\[
\frac{d}{dt}\| \nabla \log \bar\rho\|_{L^\infty} \leq C(\|\nabla (\nabla \cdot \tilde u)\|_{L^\infty} + \|\nabla \tilde u\|_{L^\infty} \|\nabla \log \bar\rho\|_{L^\infty}),
\]
and applying Gr\"onwall's lemma gives
\[
 \| \nabla \log \bar\rho\|_{L^\infty} \leq C(1 +  \| \nabla \log \bar\rho_0\|_{L^\infty}),
\]
where $C>0$ depends only on $\|\nabla (\nabla \cdot \tilde u)\|_{L^\infty}$, $\|\nabla \tilde u\|_{L^\infty}$, and $T>0$. On the other hand, similarly as in the proof of Lemma \ref{lem_ee}, we easily get
\[
\|\nabla \tilde u\|_{L^\infty} \leq C\lt(1 + \|\nabla \bar\rho\|_{L^1 \cap L^\infty}\rt)
\]
and
\[
\|\nabla (\nabla \cdot \tilde u)\|_{L^\infty} \leq \|\nabla \bar\rho\|_{L^\infty} .
\]
This completes the proof.
\end{proof}

We next provide higher-order estimates on $\nabla \log \bar\rho$.

\begin{lemma}
There exists $0 < T_* \leq T$ such that
\[
\sup_{0 \leq t \leq T_*}\|\nabla^2 \log \bar\rho(\cdot,t)\|_{L^\infty}  \leq C_1
\]
and
\[
\sup_{0 \leq t \leq T_*}\|\nabla^3 \log \bar\rho(\cdot,t)\|_{L^\infty}  \leq C_2,
\]
where $C_k>0, k=1,2$ depends only on $\|\nabla \log \bar\rho_0\|_{W^{k,\infty}}$ and $\|\nabla \bar\rho\|_{L^\infty(0,T;W^{k,1} \cap W^{k,\infty})}$.
\end{lemma}

\begin{proof}
For $i,j = 1,\dots,d$, we get
\begin{align*}
\pa_t \pa_{ij} \log \bar\rho + \tilde u \cdot \nabla \pa_{ij} \log \bar\rho &= -\pa_j \tilde u \cdot \nabla \pa_i \log \bar\rho - \nabla \cdot \pa_{ij} \tilde u - \pa_{ij} \tilde u \cdot \nabla \log \bar\rho - \pa_i \tilde u \cdot \nabla \pa_j \log \bar\rho\cr
&\quad + \Delta \pa_{ij} \log \bar\rho + 2\nabla \pa_{ij} \log \bar\rho \cdot \nabla \log \bar\rho + 2 \nabla \pa_i \log \bar\rho \cdot \nabla \pa_j \log \bar\rho.
\end{align*}
Then similarly as before, we obtain
\begin{align*}
&\frac{d}{dt}\|\nabla^2 \log \bar\rho(\cdot,t)\|_{L^\infty} \cr
&\quad \ls \|\nabla \tilde u\|_{L^\infty}\|\nabla^2 \log \bar\rho\|_{L^\infty} + \|\nabla^2 (\nabla \cdot \tilde u)\|_{L^\infty} + \|\nabla^2 \tilde u\|_{L^\infty} \|\nabla \log \bar\rho\|_{L^\infty} + \|\nabla^2 \log \bar\rho\|_{L^\infty}^2\cr
&\quad \leq C(1 + \|\nabla \bar\rho\|_{W^{1,1} \cap  W^{1,\infty}} )(1 + \|\nabla \log \bar\rho\|_{L^\infty}+ \|\nabla^2 \log \bar\rho\|_{L^\infty} + \|\nabla^2 \log \bar\rho\|_{L^\infty}^2).
\end{align*}
This gives the first assertion. 

We next estimate $ \| \nabla^3 \log \bar\rho\|_{L^\infty}$. For $i,j,k=1,\dots,d$, we find
\begin{align*}
\pa_t \pa_{ijk} \log \bar\rho + \tilde u \cdot \nabla \pa_{ijk} \log \bar\rho &= -\pa_k \tilde u \cdot \nabla \pa_{ij} \log \bar\rho - \pa_{jk}\tilde u \cdot \nabla \pa_i \log \bar\rho - \pa_j \tilde u \cdot \nabla \pa_{ik} \log \bar\rho \cr
&\quad - \nabla \cdot \pa_{ijk} \tilde u - \pa_{ijk} \tilde u \cdot \nabla \log \bar\rho - \pa_{ij} \tilde u \cdot \nabla \pa_k \log \bar\rho \cr
&\quad - \pa_{ik} \tilde u \cdot \nabla \pa_j \log \bar\rho - \pa_i \tilde u \cdot \nabla \pa_{jk} \log \bar\rho +  \Delta \pa_{ijk} \log \bar\rho \cr
&\quad + 2\nabla \pa_{ijk} \log \bar\rho \cdot \nabla \log \bar\rho + 2 \nabla \pa_{ij} \log \bar\rho \cdot \nabla \pa_k \log \bar\rho\cr
&\quad + 2 \nabla \pa_{ik} \log \bar\rho \cdot \nabla \pa_j \log \bar\rho + 2\nabla \pa_i \log \bar\rho \cdot \nabla \pa_{jk} \log \bar\rho.
\end{align*}
Then by a similar fashion as above, we obtain
\begin{align*}
\frac{d}{dt}\|\nabla^3 \log \bar\rho(\cdot,t)\|_{L^\infty} &\ls \|\nabla^3 \log \bar\rho\|_{L^\infty}\lt(\|\nabla \tilde u\|_{L^\infty} + \|\nabla^2 \log \bar\rho\|_{L^\infty} \rt)\cr
&\quad + \|\nabla^2 \log \bar\rho\|_{L^\infty}\|\nabla^2 \tilde u\|_{L^\infty} + \|\nabla \log \bar\rho\|_{L^\infty}\|\nabla^3 \tilde u\|_{L^\infty} + \|\nabla^3 (\nabla \cdot \tilde u)\|_{L^\infty}.
\end{align*}
Since $\|\nabla^3 \tilde u\|_{L^\infty} + \|\nabla^3 (\nabla \cdot \tilde u)\|_{L^\infty} \leq C\|\nabla^3 \bar\rho\|_{L^1 \cap L^\infty}$, we conclude the second assertion.
\end{proof}
The above two lemmas imply
\[
\sup_{0 \leq t \leq T_*}\|\nabla \log \bar\rho(\cdot,t)\|_{W^{2,\infty}}  \leq C,
\]
for some $T_*>0$ and $C>0$ depends only on $\|\nabla \log \bar\rho_0\|_{W^{2,\infty}}$ and $\|\nabla \bar\rho\|_{L^\infty(0,T;W^{2,1} \cap W^{2,\infty})}$.

Combining all of the above discussion yields that Lemma \ref{lem_ee} can be restated as
\begin{lemma}\label{lem_re2} There exists $C>0$ depending only on $\|\bar\rho\|_{L^\infty(0,T;W^{3,1} \cap W^{3,\infty})}$ and $\|\nabla \log \bar\rho_0\|_{W^{2,\infty}}$ such that 
\[
\intr \rho( u - \bar u) \cdot \frac{\bar e}{\bar \rho}\,dx \leq \intr \rho |u - \bar u|^2\,dx + C\lt(\intr \rho |u - \bar u|^2\,dx \rt)^{1/2}.
\]
\end{lemma}
\begin{remark}For the periodic domain case $\T^d$, the bound on $\|\bar\rho\|_{L^\infty(0,T;W^{3,1})}$ is not required, thus the constant $C > 0$ appeared in Lemma \ref{lem_re2} only depends on $\|\bar\rho\|_{L^\infty(0,T;W^{3,\infty})}$ and $\|\nabla \log \bar\rho_0\|_{W^{2,\infty}}$. Thus if one can establish a solution $\bar\rho$ in $L^\infty(0,T;H^s(\T^d))$ with $s > d/2+3$, then we have the bound $\|\bar\rho\|_{L^\infty(0,T;W^{3,\infty})} < \infty$ due to the Sobolev embedding.
\end{remark}

\section{Regularity Estimates for $\bar\rho$}\label{sec:bdd}
In this section, we provide the regularity estimates for $\bar\rho$ used in the arguments in the previous section. To this end, we first simplify the aggregation-diffusion equation \eqref{main_conti} by substituting its second equation into the first one, we have

\begin{equation}\label{Eqs}
  \pa_t\bar\rho  =  \Delta\bar\rho  +  \nabla\cdot(\bar\rho(\nabla\bar\Phi + \nabla V)), \quad
-\Delta\bar\Phi  =  \bar\rho.
\end{equation}
Since we consider the quadratic confinement $V(x)=|x|^2/2$, \eqref{Eqs} thus reduces to
\begin{equation}\label{Eqs-1}
  \pa_t\bar\rho  =  \Delta\bar\rho  +  \nabla\cdot(\bar\rho\nabla\bar\Phi + x\bar\rho), \quad
-\Delta\bar\Phi  =  \bar\rho.
\end{equation}
Taking the change of variables, motivated from \cite[Section 3]{CT98},  
\[
\bar\rho(x,t)=e^{dt}n(e^tx,\frac12(e^{2t}-1))=:e^{dt}n(\bar x,\bar t) 
\] with $\bar x=e^t x$ and $\bar t=\frac{1}{2}(e^{2t}-1)$ in \eqref{Eqs-1}, then we can compute 
\begin{align*}
\pa_t\bar\rho(x,t)&= de^{dt}n(\bar x,\bar t) +  e^{(d+1)t}x\cdot \nabla_{\bar x} n(\bar x,\bar t)   +   e^{(d+2)t}\pa_{\bar t} n(\bar x,\bar t),\cr
\nabla_{x}\bar\rho(x,t)&= e^{(d+1)t}\nabla_{\bar x} n(\bar x,\bar t),\cr
\Delta_{x}\bar\rho(x,t)&=\nabla_{x}\cdot(\nabla_{x}\bar\rho(x,t))= \nabla_{x}\cdot(e^{(d+1)t}\nabla_{\bar x} n(\bar x,\bar t)) = e^{(d+2)t}\Delta_{\bar x} n(\bar x,\bar t),
\end{align*}
and
\begin{equation*}
\nabla_{x}\cdot(x\bar\rho(x,t)) = \nabla_{x}\cdot(xe^{dt}n(\bar x,\bar t)) = de^{dt}n(\bar x,\bar t) +  xe^{(d+1)t}\cdot\nabla_{\bar x} n(\bar x,\bar t).
\end{equation*}
From the Poisson equation in \eqref{Eqs-1}, one can express $\bar\Phi(x,t)$ uniquely as $\bar\Phi(x,t)=K\star\bar\rho$, where $K$ is defined as in Remark \ref{rem-1}. 
Then we deduce that 
\begin{align}\label{Phi-1}
\begin{aligned}
\nabla_{x}\cdot\left(\bar\rho(x,t)\nabla_{x}\bar\Phi(x,t)\right)
&=\nabla_{x}\bar\rho(x,t)\cdot\nabla_{x}\bar\Phi(x,t)   +   \bar\rho(x,t)\Delta_{x}\bar\Phi(x,t) \\
&=\nabla_{x}\bar\rho(x,t)\cdot\nabla_{x}\bar\Phi(x,t)   -  \bar\rho^2(x,t).
\end{aligned}
\end{align}
The property of convolution entails that 
\begin{align*}
\nabla_{x}\bar\Phi(x,t)=\intr\nabla_{x} K(x-y)\bar\rho(y,t)\,dy
\end{align*}
and
\begin{align*}
\nabla_{x}K(x-y)
=\frac{x-y}{|x-y|^d}
=\frac{e^{-t}(\bar x-\bar y)}{e^{-dt}|\bar x-\bar y|^d}   
=e^{(d-1)t}\frac{\bar x-\bar y}{|\bar x-\bar y|^{d}}  =  e^{(d-1)t}\nabla_{\bar x}K(\bar x-\bar y)
\end{align*}
with $\bar y=e^{t}y$. Hence, we have 
\begin{align}\label{Phi-2}
\nabla_{x}\bar\Phi(x,t) = e^{(d-1)t}\intr\nabla_{\bar x}K(\bar x-\bar y)n(\bar y,\bar t)\,d\bar y = e^{(d-1)t}(\nabla_{\bar x}K\star n)(\bar x,\bar t).
\end{align}
Substituting \eqref{Phi-2} into \eqref{Phi-1}, one has
\begin{align*}
\nabla_{x}\cdot\left(\bar\rho(x,t)\nabla_{x}\bar\Phi(x,t)\right)
&=e^{(d+1)t}\nabla_{\bar x}n(\bar x,\bar t)\cdot e^{(d-1)t}(\nabla_{\bar x}K\star n)(\bar x,\bar t)
   -  e^{2dt}n^2(\bar x,\bar t) \\
& = e^{2dt}\left(\nabla_{\bar x}n(\bar x,\bar t)\cdot(\nabla_{\bar x}K\star n)(\bar x,\bar t) - n^2(\bar x,\bar t) \right) \\
&=: e^{2dt}\nabla_{\bar x}\cdot(n(\bar x,\bar t)\nabla_{\bar x}\Psi(\bar x,\bar t)),
\end{align*}
where we denote $\Psi(\bar x,\bar t):=(K\star n)(\bar x,\bar t)$. Substituting the above equalities into \eqref{Eqs-1} and using the fact that $e^{(d-2)t}=(2\bar t+1)^{(d-2)/2}$, then we obtain the equation for $n(\bar x,\bar t)$ as the following form, for simplicity, we still use the notation $x$ and $t$,
\begin{equation}\label{eqs-n}
\pa_t n  =  \Delta n   +   (2t+1)^{(d-2)/2}\nabla\cdot(n\nabla\Psi),\quad -\Delta\Psi=n
\end{equation}    
with initial data $n_0(x):=n(x,0)=\bar\rho(x,0)=\bar\rho_0(x)$. For the local-in-time existence and uniqueness of smooth solutions to \eqref{eqs-n}, we refer to \cite{CJpre} where Riesz interaction potential is considered, but it can be easily extended to the Coulomb one. With \eqref{eqs-n} at hand, $n$ can therefore be represented by the following Duhamel integral equation
\begin{equation*}
n(t)=e^{t\Delta}n_0   +  \int_0^t e^{(t-s)\Delta}\left((2s+1)^{(d-2)/2}\nabla\cdot\left(n(s)\nabla\Psi(s)\right)\right)ds
\qquad \textrm{for\,\, all} \,\,t\in(0,T)
\end{equation*}
with $\Psi=K\star n$ and $K$ defined by Remark \ref{rem-1}. Here $\{e^{t\Delta}\}_{t \geq 0}$ denotes the semigroup generated by the heat equation:
\[
\pa_t n  =  \Delta n, \quad x \in \R^d.
\]

In the proposition below, we provide some bound estimates for the equation \eqref{eqs-n}. For this, we introduce weighted norms for a function $f = f(x)$
\[
\|f\|_{L^\infty_r} := \esssup_{x \in \R^d} (1+|x|^2)^{r/2} f(x), \quad \|f\|_{W^{k,\infty}_r} := \sum_{j=0}^k \|\nabla^j f\|_{L^\infty_r}.
\]
$L^\infty_r(\R^d)$ and $W^{k,\infty}_r(\R^d)$ are functions spaces with finite corresponding norms. 
\begin{proposition}\label{prop_bdd} Let $T>0$ and $n$ be a solution to the equation \eqref{eqs-n} on the time interval $[0,T)$ with sufficient regularity. 
Suppose that the initial data $n_0$ satisfies
\[
n_0 \in W^{3,\infty}_r (\R^d) \quad \mbox{with} \quad r > d.  
\]
Then there exists $T_* \in (0,T]$ such that
\[
\sup_{0 \leq t \leq T_*} \|n(\cdot,t)\|_{W^{3,\infty}_r} < \infty.
\]
\end{proposition}
\begin{remark} It is clear that Proposition \ref{prop_bdd} implies
\[
\sup_{0 \leq t \leq T_*} \|\bar\rho(\cdot,t)\|_{W^{3,\infty}_r}< \infty.
\]
Furthermore, we deduce
\[
\bar\rho \in L^\infty(0,T_*; W^{3,1}(\R^d))
\]
since
\begin{align*}
\intr |\nabla^k \bar\rho|\,dx &= \intr \frac{1}{(1+|x|^2)^{r/2}} (1+|x|^2)^{r/2}|\nabla^k \bar\rho(x)|\,dx\cr
&\leq \|\nabla^k \bar\rho\|_{L^\infty_r} \intr  \frac{1}{(1+|x|^2)^{r/2}}\,dx \cr
&< \infty
\end{align*}
due to $r > d$, for $k=0,1,2,3$. Hence, by combining this observation and the discussion in Section \ref{sec_rem}, under the assumption that
\[
\log \bar\rho_0 \in W^{3,\infty}(\R^d) \quad \mbox{and} \quad \bar\rho_0 \in W^{3,\infty}_r(\R^d)
\]
we have $\bar\rho \in L^\infty(0,T_*; W^{3,1} \cap W^{3,\infty}(\R^d))$, which is the sufficient condition for the estimate appeared in Lemma \ref{lem_ee} (see also Lemma \ref{lem_re2}).
\end{remark}

For the rest of this section, we devote ourselves to prove Proposition \ref{prop_bdd}. We first start with the $L^p(\R^d)$-estimate of $n(x,t)$. For any $1\leq p< \infty$, multiplying the first equation in \eqref{eqs-n} by $pn^{p-1}$, integrating the resulting equation with respect to $x$, and using integration by parts, one has
\begin{align*}
\frac{d}{dt}\|n\|_{L^p}^p    +   \frac{4(p-1)}{p}\|\nabla n^{\frac{p}{2}}\|_{L^2}^2
  &=  p(2t+1)^{(d-2)/2}\intr n^{p-1}\nabla\cdot(n\nabla\Psi)\,dx   \\
  & =-p(p-1)(2t+1)^{(d-2)/2}\intr n^{p-1}\nabla n\cdot\nabla\Psi\,dx  \\
  & =-(p-1)(2t+1)^{(d-2)/2}\intr \nabla n^{p}\cdot\nabla\Psi\,dx\\
  &= (p-1)(2t+1)^{(d-2)/2}\intr  n^{p}\Delta\Psi\,dx\\
  &= -(p-1)(2t+1)^{(d-2)/2}\intr  n^{p+1}\,dx \leq 0,
\end{align*}
which implies 
\begin{equation}\label{Lp-n}
\sup_{0 \leq t \leq T} \|n(\cdot,t)\|_{L^p}\leq \|n_0\|_{L^p}  \quad \textrm{for\,\,all} \,\, 1\leq p<\infty
\end{equation}
and
\begin{equation}\label{Linfty-n}
\sup_{0 \leq t \leq T}\|n(\cdot,t)\|_{L^{\infty}}\leq \|n_0\|_{L^{\infty}}
\end{equation}
by letting $p\rightarrow \infty$ in \eqref{Lp-n}.

With the above bound estimate at hand, we first show $W^{1,\infty}_r(\R^d)$-estimate of $n$ in the following lemma.

\begin{lemma}\label{4-1} Let $T>0$ and $n$ be a solution to the equation \eqref{eqs-n} on the time interval $[0,T)$ with sufficient regularity. 
Assume $n_0 \in W^{1,\infty}_r(\R^d)$ with $r>d$. Then there exists $T_* >0$ such that the following estimates hold:
\begin{align*}
\sup_{0 \leq t \leq T}\|n(\cdot,t)\|_{L^{\infty}_r} \leq C\quad \mbox{and} \quad \sup_{0 \leq t \leq T_*}\|\nabla n(\cdot,t)\|_{L^{\infty}_r}\leq C,
\end{align*}
where $C>0$ only depends on $d,r$, $T$, and $n_0$.  
\end{lemma}
\begin{proof}
We first introduce simplified notations:
\begin{equation*}
Y(x,t) :=(1+|x|^2)^{r/2}n(x,t) \quad \mbox{and} \quad Z(x,t) :=(1+|x|^2)^{r/2}\nabla n(x,t).
\end{equation*}
In order to obtain the estimate $\|Y\|_{L^{\infty}}$, we multiply the first equation of \eqref{eqs-n} by $(1+|x^2|)^{r/2}$ to get 
\begin{align}\label{eqs-Y}
\begin{aligned}
&\pa_t Y    - (2t+1)^{(d-2)/2}\nabla\Psi\cdot\nabla Y   -  \Delta Y\cr
&\quad = -r(r-2)(1+|x|^2)^{(r-4)/2}|x|^2 n     -   rd(1+|x|^2)^{(r-2)/2}n   -  2r(1+|x|^2)^{(r-2)/2}x\cdot\nabla n   \\
&\qquad    +  (2t+1)^{(d-2)/2}Y\Delta\Psi -  r(2t+1)^{(d-2)/2}(1+|x|^2)^{(r-2)/2}x n\cdot\nabla\Psi.
\end{aligned}
\end{align}
Since we have
\begin{equation*}
-  2r(1+|x|^2)^{(r-2)/2}x\cdot\nabla n
=  -\frac{2r}{1+|x|^2}x\cdot\nabla Y    +   2r^2(1+|x|^2)^{(r-4)/2}|x|^2 n,
\end{equation*}
the equation \eqref{eqs-Y} can thus be rewritten as
\begin{align*}
\begin{aligned}
&\pa_t Y  +  \left(2r\frac{x}{1+|x|^2}  -  (2t+1)^{(d-2)/2}\nabla\Psi\right)\cdot\nabla Y   -  \Delta Y
 = R_1   +  R_2,
\end{aligned}
\end{align*}
where
\begin{equation*}
R_1 := \left(r(r+2)\frac{|x|^2}{1+|x|^2} - rd\right)(1+|x|^2)^{(r-2)/2}n
\end{equation*}
and
\begin{equation*}
R_2 := -(2t+1)^{(d-2)/2}nY     -  r(2t+1)^{(d-2)/2}(1+|x|^2)^{(r-2)/2}x n\cdot\nabla\Psi.
\end{equation*}
For $R_1$ and $R_2$, we have the following estimates:
\begin{equation*}
\|R_1\|_{L^{\infty}}\leq   C(r,d)\|(1+|x|^2)^{(r-2)/2}n\|_{L^{\infty}}   \leq C(r,d)\|Y\|_{L^{\infty}}
\end{equation*}
and
\begin{align*}
\begin{aligned}
\|R_2\|_{L^{\infty}}
&\leq C(d,T)\|n\|_{L^{\infty}}\|Y\|_{L^{\infty}}
   +   C(r,d,T)\|\nabla\Psi\|_{L^{\infty}}\|(1+|x|^2)^{(r-1)/2}n\|_{L^{\infty}}\\
&\leq  C(d,T)\|n_0\|_{L^{\infty}}\|Y\|_{L^{\infty}}
   +   C(r,d,T)\|n\|_{L^1\cap L^{\infty}}\|Y\|_{L^{\infty}}   \\
   &\leq  C(r,d,T,n_0)\|Y\|_{L^{\infty}}.
\end{aligned}
\end{align*}
So, by a similar result in \cite[Proposition A.3]{DP86}, we have the following $L^{\infty}$ estimate for $Y$
\begin{align*}
\|Y(\cdot,t)\|_{L^{\infty}} \leq \|Y(\cdot,0)\|_{L^{\infty}}  +   \int_0^t\left(\|R_1(\cdot,s)\|_{L^{\infty}}  +  \|R_2(\cdot,s)\|_{L^{\infty}}\right)ds\leq  C_1   +  C_2\int_0^t\|Y(\cdot,s)\|_{L^{\infty}}\,ds.
\end{align*}
By Gr\"onwall's inequality, we have
\begin{align}\label{Linfty-Y-1}
\|Y(\cdot,t)\|_{L^{\infty}}  \leq  C_1\left(1+C_2 T e^{C_2 T}\right)  \qquad \textrm{for a.e. } 0\leq t\leq T.
\end{align}

Now, differentiating the first equation of \eqref{eqs-n} with respect to $x$, we have
\begin{align}\label{eqs-Dn}
\begin{aligned}
&\pa_t(\nabla n)  -  (2t+1)^{(d-2)/2}\nabla\Psi\cdot\nabla(\nabla n)    -   \Delta(\nabla n)  \\
&\quad = (2t+1)^{(d-2)/2}\nabla^2\Psi\cdot\nabla n   -   2(2t+1)^{(d-2)/2}n\nabla n.
\end{aligned}
\end{align}
Then multiplying the obtained equation \eqref{eqs-Dn} by $(1+|x|^2)^{r/2}$, one has
\begin{align}\label{eqs-Z-1}
\begin{aligned}
&\pa_t Z   -  (2t+1)^{(d-2)/2}\nabla\Psi\cdot\nabla Z     -   \Delta Z\cr
&\quad =  -  r(2t+1)^{(d-2)/2}(1+|x|^2)^{(r-2)/2}(x\cdot\nabla\Psi)\nabla n  - r(r-2)(1+|x|^2)^{(r-4)/2}|x|^2 \nabla n  \\
&\qquad    -   rd(1+|x|^2)^{(r-2)/2}\nabla n  -  2r(1+|x|^2)^{(r-2)/2}x\cdot\nabla(\nabla n)  +  (2t+1)^{(d-2)/2}\nabla^2\Psi\cdot Z   \\
&\qquad      -  2(2t+1)^{(d-2)/2}n Z.
\end{aligned}
\end{align}
Analogously, we have
\begin{equation*}
-  2r(1+|x|^2)^{(r-2)/2}x\cdot\nabla(\nabla n)
  =  -\frac{2r}{1+|x|^2}x\cdot\nabla Z    +   2r^2(1+|x|^2)^{(r-4)/2}|x|^2\nabla n,
\end{equation*}
which together with \eqref{eqs-Z-1} leads to
\begin{align}\label{eqs-Z-2}
\pa_t Z +  \left(2r\frac{x}{1+|x|^2}-  (2t+1)^{(d-2)/2}\nabla\Psi\right)\cdot\nabla Z - \Delta Z =  R_3   +R_4,
\end{align}
where
\begin{equation*}
R_3 :=\left(r(r+2)\frac{|x|^2}{1+|x|^2}-rd\right)(1+|x|^2)^{(r-2)/2}\nabla n
\end{equation*}
and
\begin{align*}
R_4  &:= - r(2t+1)^{(d-2)/2}(1+|x|^2)^{(r-2)/2}(x\cdot\nabla\Psi)\nabla n +  (2t+1)^{(d-2)/2}\nabla^2\Psi\cdot Z  -  2(2t+1)^{(d-2)/2}n Z.
\end{align*}
Hence, we can easily deduce that
\begin{align}\label{K-1}
\|R_3\|_{L^{\infty}}  \leq  C(r,d)\|(1+|x|^2)^{(r-2)/2}\nabla n\|_{L^{\infty}} \leq C(r,d)\|Z\|_{L^{\infty}}
\end{align}
and
\begin{align}\label{K-2}
\begin{aligned}
\|R_4\|_{L^{\infty}}  &\leq  C(r,d,T)\|\nabla\Psi\|_{L^{\infty}}\|(1+|x|^2)^{(r-1)/2}\nabla n\|_{L^{\infty}}   \\
&\quad     +  C(d,T)\|\nabla^2\Psi\cdot Z\|_{L^{\infty}}   +   C(d,T)\|n Z\|_{L^{\infty}}    \\
&   \leq  C(r,d,T)\|n\|_{L^1\cap L^{\infty}}\|Z\|_{L^{\infty}}  +  C(d,T)\|\nabla n\|_{L^1\cap L^{\infty}}\|Z\|_{L^{\infty}}\\
&\leq C(r,d,T,n_0)\left(\|Z\|_{L^{\infty}} + \|Z\|_{L^{\infty}}^2\right).
\end{aligned}
\end{align}
Here we used
\[
\|\nabla^k \Psi\|_{L^\infty} \leq C(r,d)\|\nabla^{k-1} n\|_{L^1 \cap L^\infty}, \quad k \in \mathbb{N}
\]
for some $C > 0$. Considering the equation \eqref{eqs-Z-2} and using the above estimates \eqref{K-1} and \eqref{K-2}, we obtain that
\begin{align*}
\|Z(\cdot,t)\|_{L^{\infty}} \leq  \|Z(\cdot,0)\|_{L^{\infty}}
       +  \int_0^t\left(\|R_3(\cdot,s)\|_{L^{\infty}}   +   \|R_4(\cdot,s)\|_{L^{\infty}}\right)ds \leq   C_3  +  C_4\int_0^t\|Z(\cdot,s)\|^2_{L^{\infty}}\,ds,
\end{align*}
which further implies
\begin{align}\label{Linfty-Z}
\|Z(\cdot,t)\|_{L^{\infty}}  \leq  \frac{\|Z(\cdot,0)\|_{L^{\infty}}}{1-C_4\|Z(\cdot,0)\|_{L^{\infty}}t}  \qquad \textrm{for a.e. } 0\leq t< \frac{1}{C_4\|Z(\cdot,0)\|_{L^{\infty}}}.
\end{align}
The desired results can be easily concluded from \eqref{Linfty-Y-1} and \eqref{Linfty-Z}.
\end{proof}

Next, we will devote ourselves to estimates for higher-order derivatives. 

\begin{lemma}\label{4-2}
Let $T>0$ and $n$ be a solution to the equation \eqref{eqs-n} on the time interval $[0,T)$ with sufficient regularity. 
Assume $\nabla^2 n_0 \in W^{1,\infty}_r(\R^d)$ with $r>d$. Then there exists $T_* >0$ such that 
\[
\sup_{0 \leq t \leq T_*}\|\nabla^2 n(\cdot,t)\|_{W^{1,\infty}_r}\leq C,
\]
where $C>0$ only depends on $d,r$, $T$, and $n_0$.   
\end{lemma}
\begin{proof}
For simplicity of notation, we denote 
\[
G(x,t) :=(1+|x|^2)^{r/2}\nabla^2n(x,t) \quad \mbox{and} \quad Q(x,t) :=(1+|x|^2)^{r/2}\nabla^3n(x,t).
\]
We apply $\nabla$ to the equation \eqref{eqs-Dn} to deduce 
\begin{align}\label{eqs-D2n}
\begin{aligned}
&\pa_t(\nabla^2n) -   (2t+1)^{(d-2)/2}\nabla\Psi\cdot\nabla(\nabla^2 n)   -   \Delta(\nabla^2 n)    \cr
&\quad = (2t+1)^{(d-2)/2}\nabla^3\Psi\cdot\nabla n    +  2(2t+1)^{(d-2)/2}\nabla^2\Psi\cdot\nabla^2 n   \cr  
&\qquad  -  2(2t+1)^{(d-2)/2}\nabla n\nabla n   - 2(2t+1)^{(d-2)/2}n\nabla^2 n.
\end{aligned}
\end{align}
Multiplying \eqref{eqs-D2n} by $(1+|x|^2)^{r/2}$ and using the definition of $G(x,t)$, we get 
\begin{align}\label{eqs-G-1}
\begin{aligned}
&\pa_t G  - (2t+1)^{(d-2)/2}\nabla\Psi\cdot\nabla G  -  \Delta G   \cr
&\quad = -r(2t+1)^{(d-2)/2}\nabla\Psi\cdot(1+|x|^2)^{(r-2)/2}x\nabla^2 n \cr 
&  \qquad    -  r(r-2)(1+|x|^2)^{(r-4)/2}|x|^2\nabla^2 n  - rd(1+|x|^2)^{(r-2)/2}\nabla^2n \cr  
& \qquad  -  2r(1+|x|^2)^{(r-2)/2}x\cdot\nabla(\nabla^2n)  + (2t+1)^{(d-2)/2}\nabla^3\Psi\cdot Z  \cr  
& \qquad  +  2(2t+1)^{(d-2)/2}\nabla^2\Psi\cdot G  -    2(2t+1)^{(d-2)/2}Z\nabla n    -    2(2t+1)^{(d-2)/2}nG.
\end{aligned}
\end{align}
For the fourth term on the right-hand-side of \eqref{eqs-G-1}, we further deduce that 
\begin{align}\label{G-2}
-2r(1+|x|^2)^{(r-2)/2}x\cdot\nabla(\nabla^2n)
=-\frac{2r}{1+|x|^2}x\cdot\nabla G    +   2r^2(1+|x|^2)^{(r-4)/2}|x|^2\nabla^2 n.
\end{align}
Substituting \eqref{G-2} into \eqref{eqs-G-1} leads to 
\begin{align}\label{G-3}
\pa_t G  +   \left(2r\frac{x}{1+|x|^2} - (2t+1)^{(d-2)/2}\nabla\Psi\right)\cdot\nabla G  -  \Delta G = R_5  + R_6,
\end{align}
where 
\begin{equation*}
R_5 :=\left(r(r+2)\frac{|x|^2}{1+|x|^2}-rd\right)(1+|x|^2)^{(r-2)/2}\nabla^2 n
\end{equation*}
and 
\begin{equation*}
\begin{aligned}
R_6 &:= -r(2t+1)^{(d-2)/2}\nabla\Psi\cdot(1+|x|^2)^{(r-2)/2}x\nabla^2 n   + (2t+1)^{(d-2)/2}\nabla^3\Psi\cdot Z \cr 
&\quad  +   2(2t+1)^{(d-2)/2}\nabla^2\Psi\cdot G   -    2(2t+1)^{(d-2)/2}Z\nabla n    -    2(2t+1)^{(d-2)/2}nG.
\end{aligned}
\end{equation*}
Similarly, we have that 
\begin{align}\label{R-5}
\|R_5\|_{L^{\infty}}\leq C(r,d)\|(1+|x|^2)^{(r-2)/2}\nabla^2n\|_{L^{\infty}} 
\leq C(r,d)\|G\|_{L^{\infty}}.
\end{align}
By Lemma \ref{4-1} and \eqref{Linfty-n}, one can bound $R_6$ as 
\begin{align}\label{R-6}
\begin{aligned}
\|R_6\|_{L^{\infty}} 
&\leq C(r,d,T)\|\nabla\Psi\|_{L^{\infty}}\|(1+|x|^2)^{(r-1)/2}\nabla^2n\|_{L^{\infty}}  +  C(d,T)\|\nabla^3\Psi\|_{L^{\infty}}\|Z\|_{L^{\infty}}    \cr   
&\quad + C(d,T)\|\nabla^2\Psi\|_{L^{\infty}}\|G\|_{L^{\infty}}  +  C(d,T)\|Z\|_{L^{\infty}}\|\nabla n\|_{L^{\infty}}  +  C(d,T)\|n\|_{L^{\infty}}\|G\|_{L^{\infty}}   \cr  
&\leq C(r,d,T)\|n\|_{L^1\cap L^{\infty}}\|G\|_{L^{\infty}}   +   C(d,T)\|\nabla^2 n\|_{L^1\cap L^{\infty}}\|Z\|_{L^{\infty}}  \cr  
&\quad +   C(d,T)\|\nabla n\|_{L^1\cap L^{\infty}}\|G\|_{L^{\infty}}   +   C(d,T)\|Z\|_{L^{\infty}}^2   \cr   
&\leq C(r,d,T)\|n\|_{L^1\cap L^{\infty}}\|G\|_{L^{\infty}}  +  C(d,T)\|G\|_{L^{\infty}}\|Z\|_{L^{\infty}}  +   C(d,T)\|Z\|_{L^{\infty}}^2    \cr   
&\leq C(r,d,T,n_0)  +  C(r,d,T,n_0)\|G\|_{L^{\infty}}
\end{aligned}
\end{align}
for $t \leq T_*$, where we used the boundedness of $\|Z\|_{L^{\infty}}$. With \eqref{R-5} and \eqref{R-6} at hand, we can thus infer from the equation \eqref{G-3} that 
\begin{align*}
\|G(\cdot,t)\|_{L^{\infty}}\leq \|G(\cdot,0)\|_{L^{\infty}}  +  \int_0^t(\|R_5(\cdot,s)\|_{L^{\infty}} + \|R_6(\cdot,s)\|_{L^{\infty}})\,ds \leq C_5  +  C_6\int_0^t\|G(\cdot,s)\|_{L^{\infty}}ds,
\end{align*}
which leads to 
\begin{equation*}
\sup_{0 \leq t \leq T_*}\|G(\cdot,t)\|_{L^{\infty}} \leq C_5 (1+ C_6T_*e^{C_6T_*}).
\end{equation*}

Next, we estimate $\|Q\|_{L^{\infty}}$. To this end, we apply $\nabla$ to \eqref{eqs-D2n} to obtain that 
\begin{align}\label{eqs-D3n}
\begin{aligned}
&\partial_t(\nabla^3 n)  -  (2t+1)^{(d-2)/2}\nabla\Psi\cdot\nabla(\nabla^3n)  -  \Delta(\nabla^3n)   \cr    
&\quad = 3(2t+1)^{(d-2)/2}\nabla^2\Psi\cdot\nabla^3n   +  (2t+1)^{(d-2)/2}\nabla^4\Psi\cdot\nabla n  +  3(2t+1)^{(d-2)/2}\nabla^3\Psi\cdot\nabla^2n   \cr   
&\qquad   -  6(2t+1)^{(d-2)/2}\nabla n\nabla^2n   -  2(2t+1)^{(d-2)/2}n\nabla^3n.
\end{aligned}
\end{align}
It follows from multiplying \eqref{eqs-D3n} by $(1+|x|^2)^{r/2}$ that 
\begin{align}\label{eqs-Q-1}
\begin{aligned}
&\partial_tQ  -  (2t+1)^{(d-2)/2}\nabla\Psi\cdot\nabla Q  -  \Delta Q\cr
&\quad =  -r(2t+1)^{(d-2)/2}\nabla\Psi\cdot(1+|x|^2)^{(r-2)/2}x\nabla^3n -r(r-2)(1+|x|^2)^{(r-4)/2}|x|^2\nabla^3n  \cr   
&\quad    -   rd(1+|x|^2)^{(r-2)/2}\nabla^3n -2r(1+|x|^2)^{(r-2)/2}x\cdot\nabla(\nabla^3n)   +  3(2t+1)^{(d-2)/2}\nabla^2\Psi\cdot Q  \cr   
&\quad  +   (2t+1)^{(d-2)/2}\nabla^4\Psi\cdot Z  + 3(2t+1)^{(d-2)/2}\nabla^3\Psi\cdot G  \cr   
& \quad  -  6(2t+1)^{(d-2)/2}\nabla  n\, G   -  2(2t+1)^{(d-2)/2}nQ.
\end{aligned}	
\end{align}
Similarly as before, the forth term on he right-hand-side of \eqref{eqs-Q-1} can be rewritten as 
\begin{align}\label{Q-1}
-2r(1+|x|^2)^{(r-2)/2}x\cdot\nabla(\nabla^3n)
=-2r\frac{x}{1+|x|^2}\cdot\nabla Q  +  2r^2(1+|x|^2)^{(r-4)/2}|x|^2\nabla^3n.
\end{align}
Substituting \eqref{Q-1} into \eqref{eqs-D3n} and rearranging the resulting equality, one has 
\begin{align*}
\partial_t Q  +  \left(2r\frac{x}{1+|x|^2}-(2t+1)^{(d-2)/2}\nabla\Psi\right)\cdot\nabla Q  -  \Delta Q  =  R_7 + R_8
\end{align*}
with 
\begin{align*}
R_7 := \left(r(r+2)\frac{|x|^2}{1+|x|^2}-rd\right)(1+|x|^2)^{(r-2)/2}\nabla^3n
\end{align*}
and 
\begin{align*}
\begin{aligned}
R_8 &:= -r(2t+1)^{(d-2)/2}\nabla\Psi\cdot(1+|x|^2)^{(r-2)/2}x\nabla^3n   +  3(2t+1)^{(d-2)/2}\nabla^2\Psi\cdot Q  \cr   
& \quad  +   (2t+1)^{(d-2)/2}\nabla^4\Psi\cdot Z   + 3(2t+1)^{(d-2)/2}\nabla^3\Psi\cdot G  \cr   
& \quad -  6(2t+1)^{(d-2)/2}\nabla n\, G   -  2(2t+1)^{(d-2)/2}nQ.
\end{aligned}
\end{align*}

We can then bound $R_7$ and $R_8$ by using the obtained boundedness of $Y$, $Z$, $G$ and \eqref{Linfty-n}. Indeed, we have

\begin{align*}
\|R_7\|_{L^{\infty}} \leq C(r,d)\|(1+|x|^2)^{(r-2)/2}\nabla^3n\|_{L^{\infty}} \leq C(r,d)\|Q\|_{L^{\infty}}
\end{align*} 
and
\begin{align*}
\begin{aligned}
\|R_8\|_{L^{\infty}} &\leq C(r,d,T)\|\nabla\Psi\|_{L^{\infty}}\|(1+|x|^2)^{(r-1)/2}\nabla^3n\|_{L^{\infty}}   +   C(d,T)\|\nabla^2\Psi\|_{L^{\infty}}\|Q\|_{L^{\infty}} \cr
&\quad + C(d,T)\|\nabla^4\Psi\|_{L^{\infty}}\|Z\|_{L^{\infty}}   +  C(d,T)\|\nabla^3\Psi\|_{L^{\infty}}\|G\|_{L^{\infty}} \cr  
& \quad + C(d,T)\|\nabla n\|_{L^{\infty}}\|G\|_{L^{\infty}}   +  C(d,T)\|n\|_{L^{\infty}}\|Q\|_{L^{\infty}}  \cr  
&\leq C(r,d,T)\|n\|_{L^1\cap L^{\infty}}\|Q\|_{L^{\infty}}  +   C(d,T)\|\nabla n\|_{L^1\cap L^{\infty}}\|Q\|_{L^{\infty}}   \cr  
&\quad + C(d,T)\|\nabla^3n\|_{L^1\cap L^{\infty}}\|Z\|_{L^{\infty}}   +  C(d,T)\|\nabla^2n\|_{L^1\cap L^{\infty}}\|G\|_{L^{\infty}}   +  C(d,T)\|Z\|_{L^{\infty}}\|G\|_{L^{\infty}}  \cr  
&\leq C(r,d,T)\|n\|_{L^1\cap L^{\infty}}\|Q\|_{L^{\infty}}   +  C(d,T)\|Z\|_{L^{\infty}}\|Q\|_{L^{\infty}}   \cr  
&\quad +  C(d,T)\|G\|^2_{L^{\infty}}   +   C(d,T)\|Z\|_{L^{\infty}}\|G\|_{L^{\infty}}   \cr  
&\leq C(r,d,T,n_0)  +  C(r,d,T,n_0)\|Q\|_{L^{\infty}}
\end{aligned}
\end{align*} 
for $t \leq T_*$, where we again used the boundedness of $\|Z(\cdot,t)\|_{L^{\infty}}$ for all $t  \leq T_*$.
These estimates yield
\begin{align*}
\|Q(\cdot,t)\|_{L^{\infty}} \leq \|Q(\cdot,0)\|_{L^{\infty}}  +  \int_0^t(\|R_7(\cdot,s)\|_{L^{\infty}}  +  \|R_8(\cdot,s)\|_{L^{\infty}})\,ds 
\leq C_7  +  C_8\int_0^t\|Q(\cdot,s)\|_{L^{\infty}}\,ds,
\end{align*}
which implies that 
\[
\sup_{0 \leq t \leq T_*}\|Q(\cdot,t)\|_{L^{\infty}} \leq C_7(1+C_8 T_* e^{C_8T_*}).
\]
This completes the proof.
\end{proof}

\begin{proof}[Proof of Proposition \ref{prop_bdd}] The proof follows from a simple combination of Lemmas \ref{4-1} and \ref{4-2}.
\end{proof}

%
%
%
%
\section*{Acknowledgments}
The research of JAC was supported by the Advanced Grant Nonlocal-CPD (Nonlocal PDEs for Complex Particle Dynamics: Phase Transitions, Patterns and Synchronization) of the European Research Council Executive Agency (ERC) under the European Union’s Horizon 2020 research and innovation programme (grant agreement No. 883363) and also partially supported by EPSRC grant number EP/T022132/1.
The work of YPC is supported by NRF grant (No. 2017R1C1B2012918), POSCO Science Fellowship of POSCO TJ Park Foundation, and Yonsei University Research Fund of 2020-22-0505. YP is partially supported by the Applied Fundamental Research Program of Sichuan Province (No. 2020YJ0264).

%
%
%
%


\begin{thebibliography}{10}

\bibitem{AGS05} L. Ambrosio, N. Gigli, and G. Savar\'e, Gradient flows in metric spaces and in the space of probability measures, Lectures in Mathematics ETH Z\"urich. Birkh\"auser Verlag, Basel, 2005.

\bibitem{B93} F. Bouchut, Existence and uniqueness of a global smooth solution for the Vlasov--Poisson--Fokker--Planck system in three dimensions, J. Funct. Anal., 111, (1993), 239--258.

\bibitem{CC20} J. A. Carrillo and Y.-P. Choi, Quantitative error estimates for the large friction limit of Vlasov equation with nonlocal forces, Ann. Inst. H. Poincar\'e Anal. Non Lin\'eaire, 37, (2020), 925--954.

\bibitem{CCJpre} J. A. Carrillo, Y.-P. Choi, and J. Jung, Quantifying the hydrodynamic limit of Vlasov-type equations with alignment and nonlocal forces, Math. Models Methods Appl. Sci., 31, (2021), 327--408.

\bibitem{CCT19} J. A. Carrillo, Y.-P. Choi, and O. Tse, Convergence to equilibrium in Wasserstein distance for damped Euler equations with interaction forces, Commun. Math. Phys, 365, (2019), 329--361.

\bibitem{CCY19} J. A. Carrillo, K. Craig, and Y. Yao, Aggregation-diffusion equations: dynamics, asymptotics, and singular limits, in N. Bellomo, P. Degond, and E. Tadmor (Eds.), Active Particles Vol. II: Advances in Theory, Models, and Applications, Series: Modelling and Simulation in Science and Technology, Birkh\"auser Basel, 65--108, 2019.

\bibitem{CFGS17} J. A. Carrillo, E. Feireisl, P. Gwiazda and A. \'Swierczewska-Gwiazda, Weak solutions for Euler systems with non-local interactions, J. London Math. Soc., 95, (2017), 705--724.

\bibitem{CMV03} J. A. Carrillo, R. J. McCann, C. Villani, Kinetic equilibration rates for granular media and related equations: entropy dissipation and mass transportation estimates, Rev. Mat. Iberoamericana, 19, (2003), 971--1018.

\bibitem{CPW20} J. A. Carrillo, Y. Peng, and A. Wr\'oblewska-Kami\'nska, Relative entropy method for the relaxation limit of hydrodynamic models, Netw. Heterog. Media, 15, (2020), 369--387.

\bibitem{CS95} J. A. Carrillo and J. Soler, On the initial value problem for the Vlasov--Poisson--Fokker--Planck system with initial data in $L^p$ spaces, Math. Methods Appl. Sci., 18, (1995), 825--839.

\bibitem{CT98} J. A. Carrillo and G. Toscani, Exponential convergence toward equilibrium for homogeneous Fokker--Planck-type equations, Math. Methods Appl. Sci., 21, (1998), 1269--1286.

\bibitem{YP16} Y.-P. Choi, Global classical solutions of the Vlasov--Fokker--Planck equation with local alignment forces, Nonlinearity, 29, (2016), 1887--1916.

\bibitem{YP20} Y.-P. Choi, Large friction limit of pressureless Euler equations with nonlocal forces, preprint, arXiv:2002.01691.

\bibitem{CJpre} Y.-P. Choi and I.-J. Jeong, Classical solutions to the fractional porous medium flow, preprint, arXiv:2102.01816.

\bibitem{CTpre} Y.-P. Choi and O. Tse, Quantified overdamped limit for kinetic Vlasov--Fokker--Planck equations with singular interaction forces, preprint, arXiv:2012.00422.


\bibitem{CG07} J. F. Coulombel and T. Goudon, The strong relaxation limit of the multidimensional isothermal Euler equations, Trans. Amer. Math. Soc., 359, (2007), 637--648.

\bibitem{DP86} P. Degond, Global existence of smooth solutions for the Vlasov--Fokker--Planck equation in 1 and 2 space dimensions, Ann Sci. \'{E}cole Norm. Sup., 19, (1986), 519--542.

\bibitem{DLPSS18} M. H. Duong, A. Lamacz, M. A. Peletier, A. Schlichting, and U. Sharma, Quantification of coarse-graining error in Langevin and overdamped Langevin dynamics, Nonlinearity, 31, (2018), 4517--4566.

\bibitem{DLPS17} M. H. Duong, A. Lamacz, M. A. Peletier, and U. Sharma, Variational approach to coarse-graining of generalized gradient flows, Calc. Var. Partial Differential Equations, 56, (2017), 100.

\bibitem{FS15} R. Fetecau and W. Sun, First-order aggregation models and zero inertia limits, J. Differential Equations, 259, (2015), 6774--6802.

\bibitem{GNPS05}  T. Goudon, J. Nieto, F. Poupaud, and J. Soler, Multidimensional high-field limit of the electrostatic Vlasov--Poisson--Fokker--Planck system, J. Differential Equations, 213, (2005), 418--442.

\bibitem{Jabin00} P.-E. Jabin, Macroscopic limit of Vlasov type equations with friction, Ann. Inst. H. Poincar\'e Anal. Non Lin\'eaire, 17, (2000), 651--672.

\bibitem{JKO98} R. Jordan, D. Kinderlehrer, and F. Otto, The variational formulation of the Fokker--Planck equation, SIAM J. Math. Anal., 29, (1998), 1--17.

\bibitem{KMT13} T. Karper, A. Mellet and K. Trivisa, Existence of weak solutions to kinetic flocking models, SIAM Math. Anal. 45, (2013), 215--243.

\bibitem{KMT14} T. Karper, A. Mellet and K. Trivisa, On strong local alignment in the kinetic Cucker--Smale model, Hyperbolic conservation laws and related analysis with applications, 227--242, Springer Proc. Math. Stat., 49, Springer, Heidelberg, 2014.

\bibitem{KMT15} T. K. Karper, A. Mellet, and K. Trivisa, Hydrodynamic limit of the kinetic Cucker--Smale flocking model, Math. Models Methods Appl. Sci., 25, (2015), 131--163.


\bibitem{LT13} C. Lattanzio, A. E. Tzavaras, Relative entropy in diffusive relaxation, SIAM J. Math. Anal., 45, (2013), 1563--1584.
\bibitem{LT17} C. Lattanzio, A. E. Tzavaras, From gas dynamics with large friction to gradient flows describing diffusion theories, Comm. Partial Differential Equations, 42, (2017), 261--290.

\bibitem{NPS01} J. Nieto, F. Poupaud, and J. Soler, High-field limit for the Vlasov--Poisson--Fokker--Planck system, Arch. Ration. Mech. Anal., 158, (2001), 29--59.

\bibitem{PS00} F. Poupaud and J. Soler, Parabolic limit and stability of the Vlasov--Poisson--Fokker--Planck system, Math. Models Methods Appl. Sci., 10, (2000), 1027--1045.

\bibitem{V91} H. D. Victory, Jr, On the existence of global weak solutions for Vlasov--Poisson--Fokker--Planck systems, J. Math. Anal. Appl., 160, (1991), 525--555.

\bibitem{VO90} H. D. Victory, Jr. and B. P. O'Dwyer, On classical solutions of Vlasov--Poisson--Fokker--Planck systems, Indiana Univ. Math. J., 39, (1990), 105--156.

\bibitem{Vill02} C. Villani, A review of mathematical topics in collisional kinetic theory Handbook of Mathematical Fluid Dynamics vol I (Amsterdam: North-Holland), 2002,  pp 71--305.

\end{thebibliography}
\end{document}